\documentclass[dvipdfmx]{amsart}

\usepackage{amssymb}

\usepackage{amsfonts}
\usepackage{mathrsfs}
\usepackage{bbm}

\usepackage{hyperref}

\theoremstyle{plain}
 \newtheorem{theorem}{Theorem}[section]
 \newtheorem{lemma}[theorem]{Lemma}
 \newtheorem{proposition}[theorem]{Proposition}

\theoremstyle{definition}
 \newtheorem{definition}{Definition}[section]
 \newtheorem{remark}{Remark}[section]
 
\theoremstyle{remark}
 
\numberwithin{equation}{section}

\renewcommand{\l}{\left}
\renewcommand{\r}{\right}

\newcommand{\cleq}{\lesssim}
\newcommand{\cgeq}{\gtrsim}
\newcommand{\ceq}{\approx} 

\newcommand{\eps}{\varepsilon}
\newcommand{\R}{{\mathbb R}}

\newcommand{\Z}{{\mathbb Z}}
\newcommand{\N}{{\mathbb N}}
\def\norm#1{\left\Vert #1 \right\Vert} 
\def\tbra#1#2{\left\langle #1 , #2\right\rangle} 

\newcommand{\cG}{{\mathcal G}} 
\newcommand{\cF}{{\mathcal F}} 
\DeclareMathOperator{\esssup}{ess\sup}

\newcommand{\arXiv}[1]{\href{https://arxiv.org/abs/#1}{arXiv:#1}}

\begin{document}

\title[Lifespan estimate for DW]{The sharp estimate of the lifespan for the semilinear wave equation with time-dependent damping}
\author[M. Ikeda]{Masahiro Ikeda}
\address{Center for Advanced Intelligence Project, Riken, Nihonbashi 1-chome Mitsui Building, 15th floor, 1-4-1 Nihonbashi, Chuo-ku, Tokyo 103-0027, Japan}
\email{masahiro.ikeda@riken.jp}
\author[T. Inui]{Takahisa Inui}
\address{Department of Mathematics, Faculty of Science Division I, Tokyo University of Science\\ 1-3 Kagurazaka, Shinjuku-ku, Tokyo 162-8601, Japan}
\email{inui@rs.tus.ac.jp}
\date{}
\maketitle

\begin{abstract}
We consider the following semilinear wave equation with time-dependent damping.
\begin{align}
\tag{NLDW}
\l\{
\begin{array}{ll}
\partial_t^2 u - \Delta u + b(t)\partial_t u = |u|^{p}, & (t,x) \in [0,T) \times \R^n,
\\
u(0,x)=\eps u_0(x), u_t(0,x)=\eps u_1(x), & x \in \R^n,
\end{array}
\r.
\end{align}
where $n \in \N$, $p>1$, $\eps>0$, and $b(t)\ceq (t+1)^{-\beta}$ with $\beta \in [-1,1)$. It is known that small data blow-up occurs when $1<p< p_F$ and, on the other hand, small data global existence holds when $p>p_F$, where $p_F:=1+2/n$ is the Fujita exponent. The sharp estimate of the lifespan was well studied when $1<p< p_F$. In the critical case $p=p_F$, the lower estimate of the lifespan was also investigated. Recently, Lai and Zhou \cite{LaZh17a} obtained the sharp upper estimate of the lifespan when $p=p_F$ and $b(t)=1$. In the present paper, we give the sharp upper estimate of the lifespan when $p=p_F$ and $b(t)\ceq (t+1)^{-\beta}$ with $\beta \in [-1,1)$ by the Lai--Zhou method. 
\end{abstract}

\tableofcontents


\section{Introduction}

\subsection{Background}

We consider the following semilinear wave equation with time-dependent damping.
\begin{align}
\label{NLDW}
\tag{NLDW}
\l\{
\begin{array}{ll}
\partial_t^2 u - \Delta u + b(t)\partial_t u = |u|^{p}, & (t,x) \in [0,T) \times \R^n,
\\
u(0,x)=\eps u_0(x), u_t(0,x)=\eps u_1(x), & x \in \R^n,
\end{array}
\r.
\end{align}
where $n \in \N$, $p>1$, $u=u(t,x)$ is a real-valued unknown function, $b=b(t)$ is a given smooth positive function, $u_0=u_0(x)$ and $u_1=u_1(x)$ are given real-valued functions, and $\eps$ is a small positive parameter.

The linear damped wave equation with $b(t)=1$ is derived from a heat conduction equation with a time delay effect (see \cite{Str11}). 

When $b(t)=0$, the linear equation is the wave equation and its energy is conserved by the solution flow. On the other hand, the energy is decreasing in the case of $b(t) \neq0$ so that we call the term $b(t)u_t$ a damping term. Here, the coefficient $b(t)$ denotes the strength of the damping. In this paper, we are interested in how the damping affects the global behavior of the solution to \eqref{NLDW}. To see this, we assume that $b$ satisfies $b\in C^1((0,\infty))$ and 
\begin{align*}
b_1 (t+1)^{-\beta} \leq b(t) \leq b_2 (t+1)^{-\beta}, 
\quad
|b'(t)| \leq b_3(t+1)^{-\beta-1},
\end{align*}
for $t\geq0$ with some $\beta \in \R$ and some positive constants $b_1$, $b_2$, and $b_3$. 

The nonlinear term $|u|^p$ also affects the global behavior of the solutions. As a pioneering work, Fujita \cite{Fuj66} found the critical exponent $p_F:=1+2/n$, which is called the Fujita exponent, for the semilinear heat equation $v_t - \Delta v =v^p$ with the initial data $v(0)=v_0 \geq 0$. Namely,  if $p<p_F$, the solution blows up in finite time even if the initial data is small and, if $p>p_F$, the solution exists globally in time when the initial data is small. After that, it is proved that the solution blows up in finite time in the critical case $p=p_F$ (see Hayakawa \cite{Hay73}, Sugitani \cite{Sug75}, Kobayashi, Sirao, and Tanaka \cite{KST77}, and Weissler \cite{Wei81}).

We say that small data global existence holds when the following SDGE holds and that small data blow-up holds when the following SDBU holds. 

\begin{itemize}
\item[(SDGE)] For any initial data $(u_0,u_1)$, there exists $\eps_\ast>0$ such that the solution exists globally in time for any $\eps \in (0,\eps_\ast)$. 
\item [(SDBU)] There exists an initial data $(u_0,u_1)$ and $\eps_\ast>0$ such that the solution blows up in finite time for any $\eps \in (0,\eps_\ast)$. 
\end{itemize}

For the so-called classical damping, that is, $b(t)\equiv 1$, Li and Zhou \cite{LiZh95} proved that small data blow-up holds and obtain the sharp estimate of the lifespan when $n=1$ or $n=2$ and $p\leq p_F$. Moreover, they also proved small data global existence if $p >p_F$ when $n=1$ or $n=2$. For $n=3$, Nishihara \cite{Nis03} proved the similar result. 
For higher dimensional case, \textit{i.e.} $n\geq 4$, Todorova and Yordanov \cite{ToYo01} showed that the critical exponent is the Fujita exponent and Zhang \cite{Zha01} proved that small data blow-up holds in the critical case. 

For \eqref{NLDW} with $b(t)=(t+1)^{-\beta}$ for $\beta \in (-1,1)$, Nishihara \cite{Nis11} and Lin, Nishihara, and Zhai \cite{LNZ12} proved that the critical exponent is the Fujita exponent. D'Abbicco, Lucente, and Reissig \cite{DLR13} discussed a more general variable coefficient and the initial data. When $\beta=-1$, Wakasugi \cite{Wak17} proved that small data global existence holds if $p>p_F$ and Fujiwara, Ikeda, and Wakasugi \cite{FIW16a} showed that small data blow-up holds if $1<p<p_F$. However, it was not known whether small data blow-up holds when $\beta=-1$ and $p=p_F$. We will see that it holds by giving the sharp upper estimate of the lifespan. 

When $\beta<-1$ (overdamping) or $\beta \geq 1$ (non-effective damping), the Fujita exponent is no longer the critical exponent. When $\beta<-1$, recently, Ikeda and Wakasugi \cite{IkWap} proved that small data global existence holds for any $p>1$. Namely, there is no critical exponent. When $\beta \geq 1$, it is known that another critical exponent appears. See \cite{DLR15, LTW17a} and references therein in the non-effective damping case. 

In the present paper, we would like to obtain the sharp upper estimate of the lifespan when $b(t)=(t+1)^{-\beta}$ with $\beta\in [-1,1)$. 

\subsection{Known estimates of the lifespan}

As stated above, there are many results for the global behavior of the solutions to \eqref{NLDW}. 
In this subsection, we focus on the estimates of the lifespan.  

To state the estimates, we give the definitions of an energy solution and its lifespan.


\begin{definition}[Energy solution and its lifespan]
Let $T>0$ and $(u_0,u_1)\in H^1(\R^n) \times L^2(\R^n)$. We say that $u$ is an energy solution to \eqref{NLDW} on $[0,T)$ if $u \in C^2([0,T);H^{-1}(\R^n)) \cap C^1([0,T);L^{2}(\R^n) ) \cap C([0,T);H^{1}(\R^n))$ satisfies the initial condition $u(0)=\eps u_0$ and $u_t(0)=\eps u_1$ and satisfies the equation
\[ \partial_t^2 u - \Delta u + b(t)\partial_t u = |u|^{p},  \]
in the sense of $C([0,T):H^{-1}(\R^n))$. Moreover, the lifespan $T(\eps)=T(\eps u_0, \eps u_1)$ of an energy solution for \eqref{NLDW} is defined by 
\begin{align*}
T(\eps):= \sup \{ T\in(0,\infty]: \text{$u$ is a unique energy solution for \eqref{NLDW} on $[0,T)$}\}.
\end{align*}
\end{definition}




The known results for the estimates of the lifespan are summarized in Table 1, where we consider the coefficient $b(t)=(t+1)^{-\beta}$ with $\beta<1$. 

\begin{table}[htb]
{\renewcommand\arraystretch{2}
\begin{tabular}{|c|c|c|c|} \hline
$\beta \backslash p$
&  $\displaystyle 1< p < p_F$
&  $\displaystyle p= p_F$
\\ \hline
$\beta<-1$
& $T(\eps)=\infty$ \cite{IkWap}
& $T(\eps)=\infty$ \cite{IkWap}
\\ \hline
$\beta = -1$
& $T(\eps) \ceq \exp \l( C \eps^{-\frac{1}{\frac{1}{p-1}-\frac{n}{2}} } \r)$
& $\exp \l(\exp \l( C \eps^{-(p-1)} \r)\r) \leq T(\eps)$
\\ 
 
&  lower and upper \cite{FIW16a}
& lower \cite{FIW16a}
\\ \hline
$-1 < \beta < 1$
& $T(\eps) \ceq C \varepsilon^{-\frac{1}{\l(\frac{1}{p-1}-\frac{n}{2}\r)(1+\beta)} }$
& $\exp (C\eps^{-(p-1)}) \leq T(\eps) \leq \exp(C\eps^{-p})$ 
\\ 
\ 
& lower \cite{FIW16a} and upper \cite{IkWa15}
& lower  \cite{IkOg16,FIW16a} and upper \cite{IkOg16}
\\ 
\ 
&\  
& $ T(\eps) \leq \exp (C\eps^{-(p-1)})$ if $\beta=0$ \cite{LaZh17a}.
\\ \hline
\end{tabular}
}
\caption{Estimates of lifespan}
\end{table}

Ikeda and Wakasugi \cite{IkWap} proved small data global existence when $\beta<-1$ and $p>1$. Fujiwara, Ikeda, and Wakasugi \cite{FIW16a} obtained the lower bounds of the lifespan when $-1 \leq \beta<1$ and $1<p\leq p_F$ (see also \cite{IkOg16}). In the subcritical case $p<p_F$, the upper estimates of the lifespan were obtained by  \cite{IkWa15} ($-1<\beta<1$) and \cite{FIW16a} ($\beta=-1$). See also \cite{LiZh95, Nis03} when $\beta=0$ and $n=1,2,3$. In the critical case $p=p_F$, Ikeda and Ogawa \cite{IkOg16} obtained the upper estimate of the lifespan $T(\eps) \leq \exp(C \eps^{-p})$ when $-1<\beta<1$. However, this is not sharp. Recently, Lai and Zhou \cite{LaZh17a} obtained the sharp upper estimate of the lifespan when $\beta=0$. We will give the sharp upper estimate of the lifespan when $-1\leq \beta <1$ and $p=p_F$ in the present paper.

\subsection{Main result} 

In the present paper, we are interested in the upper estimate of the lifespan of the solution to \eqref{NLDW} with the Fujita exponent $p=p_F$. To this end, we assume that there exists $\beta \in [-1,1)\setminus \{0\}$ such that $b$ satisfies  
\begin{enumerate}
\item[(B1)] $b\in C^2((0,\infty))$,
\item[(B2)] $b(t) \ceq  (t+1)^{-\beta}$ for $t\geq0$,
\item[(B3)] $ |b'(t)| \ceq  (t+1)^{-\beta-1}$  for $t\geq0$,
\item[(B4)] $|b''(t)| \cleq  (t+1)^{-\beta-2}$ for $t\geq0$.
\end{enumerate}

\begin{remark}
\ 
\begin{enumerate}
\item The function $b(t)=(t+1)^{-\beta}$ with $\beta \in [-1,1) \setminus\{0\}$ satisfies (B1)--(B4). 
\item There is no function satisfying (B1)--(B4) with $\beta=0$. Indeed, (B3) is not compatible with (B2) when $\beta=0$. 
\item The classical damping $b(t)\equiv 1$ does not satisfy the third condition (B3), in particular, $|b'(t)| \cgeq (t+1)^{-\beta-1}$. In the classical damping case, we do not need to assume $|b'(t)| \cgeq (t+1)^{-\beta-1}$ or (B4). 
\end{enumerate}
\end{remark}

We define 
\begin{align*}
b^*:=\int_{0}^{\infty} e^{-\int_{0}^{\tau} b(s)ds}d\tau.
\end{align*}


Then, we obtain the following sharp upper estimate of the lifespan of the energy solution to \eqref{NLDW}. 
\begin{theorem}
\label{thm1.1}
Let $p=p_F$. Assume that $b$ satisfies (B1)--(B4) with some $\beta \in [-1,1)\setminus\{0\}$. 
Let $(u_0,u_1)\in (H^1(\R^n)\cap L^1(\R^n) )\times( L^2(\R^n)\cap L^1(\R^n))$ satisfy
\begin{align*}
\int_{\R^n} u_0(x) + b^* u_1(x) dx >0.
\end{align*}
Then, for sufficiently small $\eps>0$, the lifespan $T(\eps)$ of the corresponding energy solution is estimated by
\begin{align*}
T(\eps) \leq
\begin{cases}
\exp \l( \exp (C \eps^{-(p-1)}) \r)& \text{ if } \beta=-1,
\\
\exp (C \eps^{-(p-1)}) & \text{ if } \beta\in (-1,1)\setminus\{ 0\},
\end{cases}
\end{align*}
where $C$ is a positive constant independent of $\eps$. 
\end{theorem}

\begin{remark}
\ 
\begin{enumerate}
\item As stated above, the lower estimate of the lifespan is obtained as follows (see \cite{IkOg16,FIW16a} and references therein). 
\begin{align*}
T(\eps) \geq
\begin{cases}
\exp \l( \exp (C \eps^{-(p-1)}) \r)& \text{ if } \beta=-1,
\\
\exp (C \eps^{-(p-1)}) & \text{ if } \beta\in (-1,1)\setminus\{ 0\}.
\end{cases}
\end{align*}
Thus, Theorem \ref{thm1.1} gives the sharp upper estimates of the lifespan. Moreover, this upper estimate means that small data blow-up occurs when $p=p_F$ and $\beta=-1$. 
\item The estimate in Theorem \ref{thm1.1} is same as the estimate of the lifespan for the corresponding heat equation
\begin{align*}
\l\{
\begin{array}{ll}
b(t) v_{t}-\Delta v=v^{p_F}, & (t,x) \in [0,T) \times \R^n,
\\
v(0,x)=\eps v_0(x)>0, & x \in \R^n.
\end{array}
\r.
\end{align*}
Indeed, the estimate of this equation can be obtained by reducing this to the Fujita equation by setting $v(t,x)=w(\int_{0}^{t} b(s)^{-1}ds, x)$, and using the result of  \cite{LeNi92}.
On the other hand, for \eqref{NLDW}, we cannot use such reduction since it has the twice time derivative term $u_{tt}$. 
\end{enumerate}
\end{remark}

\subsection{The idea of the proof}

We use the method by Lai and Zhou \cite{LaZh17a}. They treated the classical damping case, that is, $b(t) \equiv 1$. We explain their idea briefly. They regarded \eqref{NLDW} as the semilinear heat equation with the forcing term $u_{tt}$ and obtained 
\begin{align*}
u(t)= \eps \cG(t) \ast u_0  + \int_{0}^{t} \cG(t-s) \ast \{ |u(s)|^{p} - u_{ss}(s) \} ds,
\end{align*}
where $\cG$ is the Gaussian function. By the integration by parts of the last term in the right hand side, we get the time decay from the Gaussian function. Therefore, we can regard this equation as the semilinear heat equation for large time $t>t_{\eps}$. After $t_{\eps}$, the lifespan is estimated by $ \exp (C \eps^{-(p-1)})$ by reducing an ordinary differential inequality, which appears in the paper by Li and Zhou \cite{LiZh95}. Therefore, we obtain
\begin{align*}
T(\eps) \leq t_\eps + \exp (C \eps^{-(p-1)}). 
\end{align*}
By a direct calculation, we can prove that $t_{\eps}$ depends polynomially on $\eps$, which implies $t_{\eps} \leq  \exp (C \eps^{-(p-1)})$ for small $\eps>0$, and thus we obtain the sharp upper estimate of the lifespan. 

In the present paper, we use the idea of Lai and Zhou \cite{LaZh17a}. However, we have difficulties which come from the variable coefficient $b$. To obtain a time decay estimate from the Gaussian function, we need to rewrite \eqref{NLDW} to a divergence form. To do this, we use the transformation by Lin, Nishihara, and Zhai \cite{LNZ12}. Then, we get the divergence form
\begin{align*}
(g(t)u)_{tt} - g(t) \Delta u - (g'(t)u)_t +u_t = g(t)|u|^p,
\end{align*}
where $g$ is a positive function. According to the idea of Lai and Zhou, we regard the term $(g(t)u)_{tt}  - (g'(t)u)_t$ as a forcing term. The first term has a good time decay, which comes from the Gaussian function. However, the second term has only one time derivative and thus the time decay is worse than that of the first term. To see this term as a remainder term for large time, we need to obtain the time decay from $g'$. Moreover, since we need to know how $t_{\eps}$ depends on $\eps$, we have to get the decay order of $g'$. Once we know how $t_{\eps}$ depends on $\eps$, by reducing an ordinary differential inequality, which is different from that in \cite{LaZh17a}, we obtain the sharp upper estimate of the lifespan in Theorem \ref{thm1.1}.

{\bf Notations.} We give some notations. Let $L^p(\R^n)$ denote the usual Lebesgue space equipped with the norm 
\begin{align*}
\norm{f}_{L^p}&:=\l( \int_{\R^n} |f(x)|^p dx \r)^{1/p} \text{ if } 1\leq p<\infty,
\\
\norm{f}_{L^\infty}&:= \esssup_{x\in \R^n}|f(x)|.
\end{align*}
For $s \in \Z_{\geq0}$ and $m \geq0$, we define the weighted Sobolev space $H^{s,m}(\R^n)$ by 
\begin{align*}
H^{s,m}(\R^n)&:=\{ f \in L^2(\R^n) | \norm{f}_{H^{s,m}}<\infty \},
\\
\norm{f}_{H^{s,m}}&:= \l( \sum_{|\alpha|\leq s} \int_{\R^n} (1+|x|^2)^{m} |\partial_x^\alpha f(x)|^2 dx \r)^{1/2}.
\end{align*}
In particular, let $H^{s}(\R^n):=H^{s,0}(\R^n)$. For an interval $I \subset \R$ and a Banach space $X$, we denote the space of $k$-times continuously differentiable $X$-valued function on $I$ by $C^{k}(I:X)$. We denote the Gaussian function by $\cG$, \textit{i.e.} 
\begin{align*}
\cG(t)=\cG(t,x)=(4\pi t)^{-n/2}e^{-\frac{|x|^2}{4t}}. 
\end{align*}
We set $\cG'(t):= \partial_t \cG(t,x)$ and $\cG''(t):=\partial_t^2 \cG(t,x)$. 
The symbol $A\cleq B$ (resp. $A\cgeq B$) stands for $A\leq C B$ (resp. $A\geq C B$) with some positive constant $C$. $A \ceq B$ means that $A\cleq B$ and $A\cgeq B$ hold.


\section{Proof}

\subsection{Rewriting the equation and using a test function}

First, we rewrite the equation to a heat integral equation. We give the definition of a strong solution. 

\begin{definition}
Let $T>0$ and $(u_0,u_1)\in H^2(\R^n) \times H^1(\R^n)$. We say that $u$ is a strong solution to \eqref{NLDW} on $[0,T)$ if $u \in C^2([0,T);L^2(\R^n)) \cap C^1([0,T);H^{1}(\R^n) ) \cap C([0,T);H^{2}(\R^n))$ satisfies the initial condition $u(0)=\eps u_0$ and $u_t(0)=\eps u_1$ and satisfies the equation
\[ \partial_t^2 u - \Delta u + b(t)\partial_t u = |u|^{p},  \]
in the sense of $C([0,T):L^2(\R^n))$.
\end{definition}

To apply Lai--Zhou's method, we need to rewrite \eqref{NLDW} to a divergence form. To do this, we use Lin--Nishihara--Zhai's transformation. Let $g$ satisfy
\begin{align*}
\l\{
\begin{array}{ll}
g'(t)=b(t)g(t)-1, & t\geq 0,
\\
g(0)=b^*,
\end{array}
\r.
\end{align*}
where we recall that 
\begin{align*}
b^*=\int_{0}^{\infty} e^{-\int_{0}^{\tau} b(s)ds}d\tau.
\end{align*}
The solution $g$ of this ordinary differential equation is explicitly given by
\[ g(t)= e^{\int_{0}^{t} b(s) ds} \l( \int_{0}^{\infty} e^{-\int_{0}^{\tau} b(s)ds}d\tau- \int_{0}^{t} e^{-\int_{0}^{\tau} b(s)ds}d\tau \r). \]
Then, we get the following divergence form (see \cite{LNZ12}). 
\begin{align}
(g(t)u)_{tt} - g(t) \Delta u - (g'(t)u)_t +u_t = g(t)|u|^p,
\end{align}
Namely, we get the following lemma. 
\begin{lemma}
If $u$ is a strong solution to \eqref{NLDW} on $[0,T)$, then $u$ satisfies the equation
\[ (g(t)u)_{tt} - g(t) \Delta u - (g'(t)u)_t +u_t = g(t)|u|^p, \]
in the sense of $C([0,T):L^2(\R^n))$.
\end{lemma}

We collect some properties of $g$, which are used in the sequel subsections. 

\begin{lemma}
\label{lem2.2}
The function $g$ satisfies the following properties.
\begin{enumerate}
\renewcommand{\theenumi}{\roman{enumi}}
\item $\lim_{t \to \infty} b(t)g(t)=1$, \textit{i.e.}\ $\lim_{t \to \infty} g'(t)=0$.
\item There exist positive constants $m$ and $M$ such that 
\[ m b(t)^{-1} \leq g(t) \leq M b(t)^{-1} \text{ for any } t>0.\]
Namely, $g(t)\ceq (t+1)^{\beta}$.
\item We have
\begin{align*}
\frac{g'(t)}{(b(t)^{-1})'} \to 1 \text{ as } t \to \infty.
\end{align*}
In particular, $|g'(t)| \ceq (t+1)^{\beta-1}$ for sufficiently large $t>0$.
\item We define $G$ by 
\[ G(t):= \int_{0}^{t} g(s) ds.\]
Then, $G$ is increasing and for any $t>0$,
\[ G(t)+1 \thickapprox 
\begin{cases}
\log(t+1)+1 & \beta=-1,
\\
(t+1)^{\beta+1} & \beta \in (-1,1). 
\end{cases}\]
\item We define $\Gamma$ by
\[ \Gamma(t):= \int_{0}^{t} \frac{1}{g(s)}ds. \]
Then, $\Gamma$ is increasing and for any $t>0$,
\[ \Gamma(t)+1 \thickapprox (t+1)^{-\beta+1}. \]
\end{enumerate}
\end{lemma}

We give the proof of this lemma in Appendix \ref{secA}.  

According to the idea of Lai and Zhou \cite{LaZh17a}, we regard the term $(g(t)u)_{tt}  - (g'(t)u)_t$ as a forcing term of a heat equation.
\begin{lemma} \label{lem2.3}
If $u$ is a strong solution to \eqref{NLDW} on $[0,T)$, then $u$ satisfies the following integral equation.
\begin{align}
\label{eq3.2}
u(t)&= \eps \cG(G(t)) \ast u_0 
\\ \notag
&\quad + \int_{0}^{t} \cG(G(t)-G(s)) \ast \{ g(s)|u(s)|^{p} - (g(s)u(s))_{ss} + (g'(s)u(s))_s \} ds.
\end{align}
\end{lemma}

\begin{proof}
Since $u$ is the strong solution of \eqref{NLDW}, the function $F(t):=g(t)|u(t)|^{p} - (g(t)u(t))_{tt} + (g'(t)u(t))_t$ belongs to $L^{2}(\R^n)$ for any $t \in[0,T)$. 
Taking the Fourier transform of \eqref{NLDW}, we get 
\begin{align*}
\widehat{u_t}(t) + g(t) |\xi|^2 \widehat{u} (t) =\widehat{F}(t).
\end{align*}
Since we have $\partial_{t} \widehat{u} (t)=\widehat{u_t}(t)$ in $L^2(\R^n)$, we  have the ordinary differential equation 
\begin{align*}
\partial_{t} \widehat{u}(t) + g(t) |\xi|^2 \widehat{u} (t) =\widehat{F}(t).
\end{align*}
For any $\varphi \in C_{0}^{\infty}(\R^n)$, 
we obtain
\begin{align*}
\partial_{t} \tbra{ e^{G(t)|\xi|^2} \widehat{u}(t) }{\varphi}_{L^2}
& =  \tbra{ e^{G(t)|\xi|^2} \partial_{t} \widehat{u}(t)}{\varphi}_{L^2} + \tbra{ g(t)|\xi|^2 e^{G(t)|\xi|^2} \widehat{u}(t)}{\varphi}_{L^2}.
\end{align*}
Using the equation, we get 
\begin{align*}
\partial_{t} \tbra{ e^{G(t)|\xi|^2} \widehat{u}(t) }{\varphi}_{L^2}
& = \tbra{e^{G(t)|\xi|^2} \widehat{F}(t)}{\varphi}_{L^2}.
\end{align*}
Since the right hand side is continuous on $(0,T)$, we can use the fundamental theorem of calculus. 
Therefore, we obtain
\begin{align*}
\tbra{ e^{G(t)|\xi|^2} \widehat{u}(t) -  \widehat{u}(0)}{\varphi}_{L^2}= \int_{0}^{t} \tbra{e^{G(s)|\xi|^2} \widehat{F}(s)}{\varphi}_{L^2} ds
\end{align*}
for every $t\in (0,T)$ and any $\varphi \in C_{0}^{\infty}(\R^n)$. For fixed $t \in (0,T)$, taking $\varphi(\xi) = e^{-G(t)|\xi|^2} \phi (\xi)$ where $\phi \in C_{0}^{\infty}(\R^n)$ is arbitrary, we obtain
\begin{align*}
\tbra{ e^{G(t)|\xi|^2} \widehat{u}(t) -  \widehat{u}(0)}{e^{-G(t)|\xi|^2} \phi }_{L^2}= \int_{0}^{t} \tbra{e^{G(s)|\xi|^2} \widehat{F}(s)}{e^{-G(t)|\xi|^2} \phi }_{L^2} ds,
\end{align*}
for any $t \in (0,T)$. 
Thus, we get
\begin{align*}
\tbra{  \widehat{u}(t) - e^{-G(t)|\xi|^2} \widehat{u}(0)}{ \phi }_{L^2}= \int_{0}^{t} \tbra{e^{-(G(t)-G(s))|\xi|^2} \widehat{F}(s)}{\phi }_{L^2} ds,
\end{align*}
for any $t \in (0,T)$. Since we have
\begin{align*}
 \int_{0}^{t} \tbra{e^{-(G(t)-G(s))|\xi|^2} \widehat{F}(s)}{\phi }_{L^2} ds=  \tbra{  \int_{0}^{t} e^{-(G(t)-G(s))|\xi|^2} \widehat{F}(s) ds }{\phi }_{L^2},
\end{align*}
by the Fubini--Tonelli theorem, we get 
\begin{align*}
\tbra{  \widehat{u}(t) - e^{-G(t)|\xi|^2} \widehat{u}(0)}{ \phi }_{L^2}
=  \tbra{  \int_{0}^{t} e^{-(G(t)-G(s))|\xi|^2} \widehat{F}(s) ds }{\phi }_{L^2}.
\end{align*}
Since $\phi \in C_{0}^{\infty}(\R^n)$ is arbitrary, we obtain
\begin{align*}
\widehat{u}(t) - e^{-G(t)|\xi|^2} \widehat{u}(0)
=  \int_{0}^{t} e^{-(G(t)-G(s))|\xi|^2} \widehat{F}(s) ds,
\end{align*}
for almost every $\xi\in\R^n$. We note that each term belongs to $L^2(\R^n)$ for any $t \in (0,T)$. Taking the inverse Fourier transform, we obtain
\begin{align*}
u(t) - ( e^{-G(t)|\xi|^2} \widehat{u}(0))^{\vee}= \cF^{-1} \l( \int_{0}^{t} e^{-(G(t)-G(s))|\xi|^2} \widehat{F}(s) ds\r). 
\end{align*}
By the Parseval identity and the Fubini--Tonelli theorem, we get
\begin{align*}
u(t) - ( e^{-G(t)|\xi|^2} \widehat{u}(0))^{\vee}=\int_{0}^{t} ( e^{-(G(t)-G(s))|\xi|^2} \widehat{F}(s))^{\vee} ds. 
\end{align*}
Now, we have $( e^{-G(t)|\xi|^2}\widehat{f})^{\vee}(x) = \cG(G(t)) \ast f $ for $f \in L^2(\R^n)$. The proof is completed. 
\end{proof}



Choosing the Gaussian function as a test function, we get the following proposition. 

\begin{proposition} \label{prop2.4}
The energy solution $u$ to \eqref{NLDW} satisfies 
\begin{align}
\label{eq3.5}
&\int_{\R^n} e^{-\frac{|x|^2}{4(G(t)+1)}} u(t,x) dx
+g(t) \int_{\R^n} e^{-\frac{|x|^2}{4(G(t)+1)}} \partial_t u(t,x) dx
\\ \notag
&= \eps (4\pi (G(t)+1))^{n/2}  \int_{\R^n} \cG(2G(t)+1,x) (u_0+g(0)u_1)(x)  dx
\\ \notag
& \quad + (4\pi (G(t)+1))^{n/2} \int_{0}^{t} g(s) \int_{\R^n} \cG(2G(t)-G(s)+1,x)  |u(s,x)|^{p}  dx ds 
\\ \notag
& \quad - (4\pi (G(t)+1))^{n/2} \int_{0}^{t}g(s)^2  \int_{\R^n}   \cG'(2G(t)-G(s)+1,x)  \partial_su (s,x) dx ds,
\end{align}
 for any $t \in (0,T(\eps))$.
\end{proposition}

To prove this proposition, we use the following approximation argument. 
\begin{lemma} \label{lem2.5}
Let $(u_0,u_1) \in H^1(\R^n) \times L^2(\R^n)$ and $u$ be the energy solution to \eqref{NLDW} on $[0,T(\eps))$ with the initial data $(u_0,u_1)$. Then,  there exists a sequence $\{(u_0^k,u_1^k)\}_{k \in \N} \subset H^2(\R^n) \times H^1(\R^n)$ such that $(u_0^k,u_1^k) \to (u_0,u_1)$ in $H^1(\R^n) \times L^2(\R^n)$ as $k\to \infty$ and 
\begin{align*}
\norm{u^k-u}_{L^{\infty}(K:H^1(\R^n))} + \norm{\partial_{t} u^k- \partial_{t} u}_{L^{\infty}(K:L^2(\R^n))} \to 0 \text{ as } k \to \infty,
\end{align*}
for any compact set $K \subset [0,T(\eps))$, where $u^k$ is the strong solution to \eqref{NLDW} with the initial data $(u_0^k,u_1^k)$.
\end{lemma}

\begin{proof}
It is enough to prove persistency of the regularity and continuous dependence on initial data since the statement can be obtained by combining them. 
First, we show persistency. Namely, the lifespan of the strong solution $T_s$ equals to the lifespan $T$ of the energy solution if $(u_0,u_1) \in H^2(\R^n) \times H^1(\R^n)$. It is easy to show $T_s \leq T$ so that we prove $T_s \geq T$ by a contradiction argument. We suppose that $T_s < T$. Then, we have $T_s<\infty$ and thus \begin{align}
\label{eq2.4}
\lim_{t \to T_s} \norm{ ( u(t), u_t (t)) }_{H^2\times H^1}=\infty. 
\end{align}
Since $T_s < T$, we have 
\begin{align*}
M:= \sup_{0 < t < T_s} \norm{ ( u(t), u_t (t)) }_{H^1 \times L^2}<\infty.
\end{align*}
Take $t_0 \in (0,T_s)$ such that $T_s -t_0 <(2C_0 M^{p-1})^{-1}$. Then, by the energy estimate, we have
\begin{align}
\label{eq2.5}
\sup_{t_0 \leq t \leq \tau}  \norm{ ( u(t), u_t (t)) }_{H^2\times H^1} 
\cleq   \norm{ ( u(t_0), u_t (t_0)) }_{H^2\times H^1} 
+ \int_{t_0}^{\tau} \norm{ |u(t)|^{p} }_{H^1} dt,
\end{align}
for any $\tau \in (t_0, T_s)$. 
Now, since $p\leq n/(n-2)$, by the Sobolev inequality, we obtain
\begin{align*}
\norm{ |u(t)|^{p} }_{L^2}
= \norm{ u(t) }_{L^{2p}}^{p}
\leq \norm{u(t)}_{H^1}^{p}.
\end{align*}
Moreover, by the H\"{o}lder inequality and the Sobolev inequality, we also obtain
\begin{align*}
\norm{ |u(t)|^{p} }_{\dot{H}^1}
& \cleq \norm{ |u(t)|^{p-1} |\nabla u(t)| }_{L^2}
\\
& \leq  \norm{ |u(t)|^{p-1} }_{L^n} \norm{\nabla u(t) }_{L^{\frac{2n}{n-2}}}
\\
& \cleq  \norm{ u(t) }_{L^{n(p-1)}}^{p-1} \norm{u(t) }_{H^2}
\\
& \cleq  \norm{ u(t) }_{H^1}^{p-1} \norm{u(t) }_{H^2},
\end{align*}
where we use $n(p-1) \leq 2n/(n-2)$. Thus, combining them with \eqref{eq2.5}, we get
\begin{align*}
\sup_{t_0 \leq t \leq \tau}  \norm{ ( u(t), u_t (t)) }_{H^2\times H^1} 
& \leq  C \norm{ ( u(t_0), u_t (t_0)) }_{H^2\times H^1} 
\\
& \quad +C_0 (\tau-t_0) \sup_{t_0 \leq t \leq \tau} \l( \norm{ u(t) }_{H^1}^{p-1} \norm{u(t) }_{H^2}\r)
\\
& \leq C \norm{ ( u(t_0), u_t (t_0)) }_{H^2\times H^1} 
\\
& \quad +C_0 M^{p-1} (T_s-t_0) \sup_{t_0 < t < \tau} \norm{ ( u(t), u_t (t)) }_{H^2\times H^1},
\end{align*}
for any $\tau \in (t_0, T_s)$. By the definition of $t_0$, we have
\begin{align*}
\sup_{t_0 \leq t \leq \tau}  \norm{ ( u(t), u_t (t)) }_{H^2\times H^1} 
 \leq C \norm{ ( u(t_0), u_t (t_0)) }_{H^2\times H^1},
\end{align*}
for any $\tau \in (t_0, T_s)$. Taking the limit $\tau \to T_s$, this contradicts \eqref{eq2.4}. 
Next, we show continuous dependence on initial data. Let $\{(u_0^k,u_1^k)\}_{k\in \N} \subset H^1(\R^n) \times L^2(\R^n)$ satisfy $(u_0^k,u_1^k) \to (u_0,u_1)$ in $H^1(\R^n) \times L^2(\R^n)$ as $k\to \infty$. We have $\norm{(u_0^k,u_1^k)}_{H^1 \times L^2} \leq 2\norm{(u_0,u_1)}_{H^1 \times L^2}$ for large $k$. 
By the energy estimate, there exists $T_1=T_1(\norm{(u_0,u_1)}_{H^1 \times L^2} )$ such that
\begin{align*}
\sup_{0\leq t \leq T_1}\norm{(u(t),u_t(t))}_{H^1 \times L^2} + \sup_{k\geq K} \sup_{0\leq t \leq T_1}\norm{(u^k(t),u_t^k(t))}_{H^1 \times L^2} 
\cleq \norm{(u_0,u_1)}_{H^1 \times L^2}.
\end{align*}
This and the energy estimate give 
\begin{align*}
&\norm{(u(t),u_t(t))-(u^k(t),u_t^k(t))}_{H^1 \times L^2}
\\
&\quad \cleq  \norm{(u_0,u_1)-(u_0^k,u_1^k)}_{H^1 \times L^2} + \int_{0}^{T_1}  \norm{(u(t),u_t(t))-(u^k(t),u_t^k(t))}_{H^1 \times L^2} dt.
\end{align*}
By the Gronwall inequality, we obtain 
\begin{align*}
\sup_{0\leq t \leq T_1}\norm{(u(t),u_t(t))-(u^k(t),u_t^k(t))}_{H^1 \times L^2} \to 0 
\end{align*}
as $k \to \infty$. Iterating this to cover any compact subset of $(0, T)$, we finish the proof.
\end{proof}

\begin{proof}[Proof of Proposition \ref{prop2.4}] 
It is enough to show that the strong solution satisfies \eqref{eq3.5} since we find that the energy solution satisfies \eqref{eq3.5} for any $t \in (0,T(\eps))$ by Lemma \ref{lem2.5} and the approximation argument. Let $u$ be the strong solution to \eqref{NLDW}. By Lemma \ref{lem2.3}, the strong solution satisfies \eqref{eq3.2}. Multiplying \eqref{eq3.2} by $(4\pi (G(t)+1))^{n/2} \cG(G(t)+1,x)$ and integrating it on $\R^n$, we get
\begin{align}
\label{eq2.6}
&\int_{\R^n} e^{-\frac{|x|^2}{4(G(t)+1)}} u(t,x) dx
\\ \notag
&= \eps (4\pi (G(t)+1))^{n/2} \int_{\R^n}  \cG(G(t)+1,x)( \cG(G(t)) \ast u_0 )(x) dx
\\ \notag
&\quad +(4\pi (G(t)+1))^{n/2}  \int_{\R^n}  \cG(G(t)+1,x) \int_{0}^{t}  \cG(G(t)-G(s)) \ast g(s)|u(s)|^{p}  ds dx
\\ \notag
&\quad -(4\pi (G(t)+1))^{n/2}  \int_{\R^n}  \cG(G(t)+1,x) \int_{0}^{t}  \cG(G(t)-G(s)) \ast   (g(s)u(s))_{ss}  ds dx
\\ \notag
&\quad +(4\pi (G(t)+1))^{n/2}  \int_{\R^n}  \cG(G(t)+1,x) \int_{0}^{t}  \cG(G(t)-G(s)) \ast  (g'(s)u(s))_s ds dx.
\end{align}
By the semigroup property of the Gaussian, \textit{i.e.} $\cG(s) \ast \cG(t) = \cG(s+t)$ for any $s,t> 0$, and the Fubini--Tonelli theroem, we obtain
\begin{align*}
& \int_{\R^n}  \cG(G(t)+1,x) \cG(G(t)) \ast u_0(x)  dx
 =  \int_{\R^n}   \cG(2G(t)+1,x)  u_0(x)   dx,
\end{align*}
and also get
\begin{align*}
& \int_{\R^n} \cG(G(t)+1,x)  \{\cG(G(t)-G(s))\ast |u(s)|^{p} \}(x) dx
\\
& = \int_{\R^n} \cG(2G(t)-G(s)+1,x)  |u(s,x)|^{p}  dx.
\end{align*}
Morover, we have
\begin{align*}
& \int_{\R^n}  \cG(G(t)+1,x) \int_{0}^{t}  \cG(G(t)-G(s)) \ast   (g(s)u(s))_{ss}  ds dx
\\
&=\int_{\R^n}  \cG(G(t)+1,x) \int_{0}^{t}  \cG(G(t)-G(s)) \ast   (g'(s)u(s))_{s}  ds dx
\\
&\quad + \int_{\R^n}  \cG(G(t)+1,x) \int_{0}^{t}  \cG(G(t)-G(s)) \ast   (g(s)u_s(s))_{s}  ds dx,
\end{align*}
and
\begin{align*}
&\int_{\R^n}  \cG(G(t)+1,x) \int_{0}^{t}  \cG(G(t)-G(s)) \ast   (g(s)u_s(s))_{s}  ds dx
\\
&= \int_{0}^{t}  \int_{\R^n}  \cG(2G(t)-G(s)+1,x) (g(s)u_s(s))_s(x)dx ds
\\
&=  g(t) \int_{\R^n}  \cG(G(t)+1,x) u_t(t,x)dx - \eps g(0) \int_{\R^n}  \cG(2G(t)+1,x) u_1(x)dx 
\\
&\quad + \int_{0}^{t}   \int_{\R^n} g(s)^2 \cG'(2G(t)-G(s)+1,x) u_s(s,x)dx ds.
\end{align*}
Therefore, \eqref{eq2.6} implies \eqref{eq3.5} for the strong solution. 
\end{proof}


\subsection{Estimate of each term}

To obtain an ordinary differential inequality, we give estimates for each term of \eqref{eq3.5}. 


We set 
\begin{align*}
A(t) & := \int_{\R^n} e^{-\frac{|x|^2}{4(G(t)+1)}} u(t,x) dx,
\\
B(t) & := g(t) \int_{\R^n} e^{-\frac{|x|^2}{4(G(t)+1)}} \partial_t u(t,x) dx,
\\
C(t) & := \eps (4\pi (G(t)+1))^{n/2}  \int_{\R^n} \cG(2G(t)+1,x) (u_0+g(0)u_1)(x)  dx,
\\
D(t) & :=  (4\pi (G(t)+1))^{n/2} \int_{0}^{t} g(s) \int_{\R^n} \cG(2G(t)-G(s)+1,x)  |u(s,x)|^{p}  dx ds,
\\
E(t) & := - (4\pi (G(t)+1))^{n/2} \int_{0}^{t}g(s)^2  \int_{\R^n}   \cG'(2G(t)-G(s)+1,x)  \partial_su (s,x) dx ds.
\end{align*}
Namely, \eqref{eq3.5} is $A+B=C+D+E$. 

\subsubsection{Estimate of $A$} 

First, we estimate $A$. Let $p'$ denote the H\"{o}lder conjugate of $p$ \textit{i.e.}\ $p'=p/(p-1)$. By the H\"{o}lder inequality, we obtain
\begin{align*}
A(t) &= \int_{\R^n} e^{-\frac{|x|^2}{4(G(t)+1)}} u(t,x) dx
\\
& \leq  \int_{\R^n} e^{-\frac{|x|^2}{4(G(t)+1)}} |u(t,x)| dx 
\\
& \leq \l( \int_{\R^n} e^{-\frac{|x|^2}{4(G(t)+1)}} dx \r)^{1/p'} \l(\int_{\R^n} e^{-\frac{|x|^2}{4(G(t)+1)}} |u(t,x)|^{p} dx\r)^{1/p} 
\\
& = ( 4\pi (G(t)+1) )^{\frac{n}{2p'}} \l(\int_{\R^n} e^{-\frac{|x|^2}{4(G(t)+1)}} |u(t,x)|^{p} dx\r)^{1/p}.
\end{align*}

We denote the right hand side by $F$, \textit{i.e.}\ 
\begin{align*}
F(t):= ( 4\pi (G(t)+1) )^{\frac{n}{2p'}} \l(\int_{\R^n} e^{-\frac{|x|^2}{4(G(t)+1)}} |u(t,x)|^{p} dx\r)^{1/p}.
\end{align*}


\subsubsection{Estimate of $B$} 

Secondly, we estimate $B$. By the Leibnitz rule, we have
\begin{align*}
B(t) & = g(t) \int_{\R^n} e^{-\frac{|x|^2}{4(G(t)+1)}} \partial_t u(t,x) dx
\\
& = g(t)A'(t) - \frac{ g(t)^2}{G(t)+1} \int_{\R^n} \frac{ |x|^2}{4(G(t)+1)}  e^{-\frac{|x|^2}{4(G(t)+1)}} u(t,x) dx.
\end{align*}
Here, by the H\"{o}lder inequality, we have
\begin{align*}
&\l| \frac{ g(t)^2}{G(t)+1} \int_{\R^n} \frac{ |x|^2}{4(G(t)+1)}  e^{-\frac{|x|^2}{4(G(t)+1)}} u(t,x) dx \r|
\\
&\leq \frac{ g(t)^2}{G(t)+1} \int_{\R^n} \frac{ |x|^2}{4(G(t)+1)}  e^{-\frac{|x|^2}{4(G(t)+1)}} |u(t,x)| dx
\\
&\leq  \frac{ g(t)^2}{G(t)+1} \l( \int_{\R^n} \l(\frac{ |x|^2}{4(G(t)+1)} \r)^{p'} e^{-\frac{|x|^2}{4(G(t)+1)}}  dx\r) ^{1/p'}\l(\int_{\R^n} e^{-\frac{|x|^2}{4(G(t)+1)}} |u(t,x)|^{p} dx\r)^{1/p}
\\
&\ceq  \frac{g(t)^2}{G(t)+1} F(t).
\end{align*}
Therefore, we get 
\begin{align*}
B(t) \leq g(t)A'(t) +C_{n,p} \frac{g(t)^2}{G(t)+1} F(t).
\end{align*}


\subsubsection{Estimate of $D$} 

Thirdly, we estimate $D$. Since $0\leq G(s) < G(t)$ for $t>s\geq 0$, we have 
\begin{align*}
0 \leq G(s)+1\leq  2G(t)-G(s)+1 \leq 2G(t)+1 \leq 2(G(t)+1).
\end{align*}
Therefore, we get
\begin{align*}
D(t) & =  (4\pi (G(t)+1))^{n/2} \int_{0}^{t} g(s) \int_{\R^n} \cG(2G(t)-G(s)+1,x)  |u(s,x)|^{p}  dx ds
\\
&= \int_{0}^{t}  \l(\frac{G(t)+1}{2G(t)-G(s)+1}\r)^{n/2}g(s) \int_{\R^n} e^{-\frac{|x|^2}{4(2G(t)-G(s)+1)}}  |u(s,x)|^{p}  dx ds
\\
&\geq 2^{-\frac{n}{2}}  \int_{0}^{t}  g(s) \int_{\R^n} e^{-\frac{|x|^2}{4(G(s)+1)}}  |u(s,x)|^{p}  dx ds
\\
&= 2^{-\frac{n}{2}}  \int_{0}^{t}  g(s) ( 4\pi (G(s)+1) )^{-\frac{np}{2p'}} F(s)^p ds
\\
&= 2^{-\frac{n}{2}} ( 4\pi  )^{-\frac{np}{2p'}} \int_{0}^{t} \frac{g(s)F(s)^p}{(G(s)+1) ^{\frac{n}{2}(p-1)}}  ds,
\end{align*}
where we note that $p/p'=p-1$.


\subsubsection{Estimate of $E$} 
We give an estimate of $E$. 
By the integration by parts, we obtain
\begin{align*}
E(t) 
&= - (4\pi (G(t)+1))^{n/2}  \int_{\R^n} [g(s)^2   \cG'(2G(t)-G(s)+1,x)  u (s,x)]_{s=0}^{s=t}dx
\\
&\quad +(4\pi (G(t)+1))^{n/2}  \int_{\R^n} \int_{0}^{t} \partial_s \{ g(s)^2   \cG'(2G(t)-G(s)+1,x) \} u (s,x) dsdx
\\
&= - (4\pi (G(t)+1))^{n/2}  \int_{\R^n} g(t)^2   \cG'(G(t)+1,x)  u (t,x)dx
\\
&\quad + \eps (4\pi (G(t)+1))^{n/2}  \int_{\R^n} g(0)^2   \cG'(2G(t)+1,x)  u_0 (x)dx
\\
&\quad +(4\pi (G(t)+1))^{n/2}  \int_{\R^n} \int_{0}^{t} 2 g(s)g'(s)  \cG'(2G(t)-G(s)+1,x)  u (s,x) dsdx
\\
&\quad -(4\pi (G(t)+1))^{n/2}  \int_{\R^n} \int_{0}^{t} g(s)^3  \cG''(2G(t)-G(s)+1,x)  u (s,x) dsdx
\\
&=:E_1(t)+E_2(t)+E_3(t)+E_4(t).
\end{align*}
By a simple calculation, we have
\begin{align*}
\cG'(t)
&=\l(-\frac{n}{2t}+\frac{|x|^2}{4t^2}\r)\cG(t).
\end{align*}
The first term is estimated as follows. 
\begin{align*}
|E_1(t)|
&\leq (4\pi (G(t)+1))^{n/2}   \frac{ng(t)^2}{2(G(t)+1)}  \int_{\R^n} \cG(G(t)+1,x) |u (t,x)|dx
\\
& \quad +(4\pi (G(t)+1))^{n/2}  \frac{g(t)^2}{G(t)+1} \int_{\R^n}   \frac{|x|^2}{4(G(t)+1)}  \cG(G(t)+1,x) |u (t,x)|dx
\\
&\leq   \frac{ng(t)^2}{2(G(t)+1)}  \int_{\R^n} e^{-\frac{|x|^2}{4(G(t)+1)}} |u (t,x)|dx
\\
& \quad +  \frac{g(t)^2}{G(t)+1} \int_{\R^n}   \frac{|x|^2}{4(G(t)+1)} e^{-\frac{|x|^2}{4(G(t)+1)}} |u (t,x)|dx.
\end{align*}
In the same way as the estimates of $A$ and $B$, we obtain
\begin{align*}
|E_1(t)|
&\leq C_{n,p} \frac{g(t)^2}{G(t)+1} F(t).
\end{align*}

The second term $E_2$ is related to the initial data. The estimate of $E_2$ is considered in Section \ref{sec2.3}. Now, we only note that the following equality holds. 
\begin{align*}
E_2(t)
&=-\eps \l( \frac{G(t)+1}{2G(t)+1}\r)^{n/2}  \frac{n g(0)^2}{2(2G(t)+1)}  \int_{\R^n}   e^{-\frac{|x|^2}{4(2G(t)+1)}} u_0 (x)dx
\\
&\quad +\eps \l( \frac{G(t)+1}{2G(t)+1}\r)^{n/2}  \frac{g(0)^2 }{2G(t)+1} \int_{\R^n}  \frac{|x|^2}{4(2G(t)+1)} e^{-\frac{|x|^2}{4(2G(t)+1)}} u_0 (x)dx.
\end{align*}

The third term $E_3$ is estimated as follows. Since $2G(t)-G(s)+1 > G(t)+1$ for $t > s\geq0$, we have
\begin{align*}
&|E_3(t)|
\\
& \leq \frac{n(4\pi)^{n/2}}{2} ( G(t)+1)^{n/2-1}  \int_{\R^n} \int_{0}^{t}  g(s)|g'(s)|  \cG(2G(t)-G(s)+1,x)  |u (s,x)| dsdx 
\\
&\quad +2(4\pi)^{n/2} (G(t)+1)^{n/2-1}  \int_{\R^n} \int_{0}^{t} g(s)|g'(s)| \frac{|x|^2 \cG(2G(t)-G(s)+1,x)  }{4(2G(t)-G(s)+1)} |u (s,x)| dsdx
\\
&=:E_{31}(t) + E_{32}(t).
\end{align*}
By the H\"{o}lder inequality and the Young inequality, we obtain
\begin{align*}
&E_{31}(t)
\\
&\ceq ( G(t)+1)^{n/2-1}  \int_{\R^n} \int_{0}^{t}  g(s)|g'(s)| \cG(2G(t)-G(s)+1,x)  |u (s,x)| dsdx 
\\
&\leq ( G(t)+1)^{n/2-1} \l( \int_{\R^n} \int_{0}^{t}  g(s)|g'(s)|^{p'} \cG(2G(t)-G(s)+1,x)  dsdx\r)^{1/p'}
\\
& \quad \times  \l(\int_{\R^n} \int_{0}^{t}  g(s)\cG(2G(t)-G(s)+1,x)  |u (s,x)|^p dsdx \r)^{1/p}
\\
&\ceq ( G(t)+1)^{\frac{n}{2p'}-1} \l( \int_{0}^{t}  g(s)|g'(s)|^{p'}  ds\r)^{1/p'} D(t)^{1/p}
\\
&\leq \frac{1}{8}D(t) + C_{n,p} ( G(t)+1)^{\frac{n}{2}-p'} \l( \int_{0}^{t}  g(s)|g'(s)|^{p'}  ds\r).
\end{align*}
In the same manner, we get
\begin{align*}
& E_{32}(t)
\\
& \ceq  (G(t)+1)^{n/2-1}  \int_{\R^n} \int_{0}^{t} g(s)|g'(s)| \frac{|x|^2 \cG(2G(t)-G(s)+1,x)  }{4(2G(t)-G(s)+1)} |u (s,x)| dsdx
\\
& \leq  (G(t)+1)^{n/2-1}  \l(\int_{\R^n} \int_{0}^{t}  g(s)\cG(2G(t)-G(s)+1,x)  |u (s,x)|^p dsdx \r)^{1/p}
\\
& \quad  \times  \l( \int_{\R^n} \int_{0}^{t} g(s)|g'(s)|^{p'} \l( \frac{|x|^2}{4(2G(t)-G(s)+1)} \r)^{p'}\cG(2G(t)-G(s)+1,x) dsdx\r)^{1/p'}
\\
& \ceq  (G(t)+1)^{\frac{n}{2p'}-1}  D(t)^{1/p} \l( \int_{0}^{t} g(s)|g'(s)|^{p'}ds\r)^{1/p'}
\\
& \leq  \frac{1}{8}D(t) + C_{n,p} (G(t)+1)^{\frac{n}{2}-p'}  \l( \int_{0}^{t} g(s)|g'(s)|^{p'}ds\r).
\end{align*}
Therefore, we conclude that
\begin{align*}
|E_3(t)| \leq  \frac{1}{4} D(t) +C_{n,p} (G(t)+1)^{\frac{n}{2}-p'}  \l( \int_{0}^{t} g(s)|g'(s)|^{p'}ds\r).
\end{align*}
At last, we consider $E_4$. Now, an easy calculation gives us 
\begin{align*}
\cG''(t)
& =\l(\frac{2n+n^2}{4t^2}-\frac{(n+2)|x|^2}{4t^3}+\frac{|x|^4}{16t^4}\r)\cG(t).
\end{align*}
Therefore, we have
\begin{align*}
E_4(t)
&=-(4\pi (G(t)+1))^{n/2}  \int_{\R^n} \int_{0}^{t} g(s)^3 \frac{(2n+n^2)\cG(2G(t)-G(s)+1,x) }{4(2G(t)-G(s)+1)^2}  u (s,x) dsdx
\\
&\quad +(4\pi (G(t)+1))^{n/2}  \int_{\R^n} \int_{0}^{t} g(s)^3 \frac{(n+2)|x|^2\cG(2G(t)-G(s)+1,x) }{4(2G(t)-G(s)+1)^3}  u (s,x) dsdx
\\
&\quad -(4\pi (G(t)+1))^{n/2}  \int_{\R^n} \int_{0}^{t} g(s)^3 \frac{|x|^4\cG(2G(t)-G(s)+1,x) }{16(2G(t)-G(s)+1)^4}  u (s,x) dsdx.
\end{align*}
Since $2G(t)-G(s)+1 > G(t)+1$ for $t > s\geq0$, we have
\begin{align*}
|E_4(t)|
&\cleq  (G(t)+1)^{n/2-2}  \int_{\R^n} \int_{0}^{t} g(s)^3 \cG(2G(t)-G(s)+1,x) |u (s,x)| dsdx
\\
& \quad + (G(t)+1)^{n/2-2}  \int_{\R^n} \int_{0}^{t} g(s)^3 \frac{|x|^2\cG(2G(t)-G(s)+1,x) }{4(2G(t)-G(s)+1)}  |u (s,x) |dsdx
\\
&\quad +(G(t)+1)^{n/2-2}  \int_{\R^n} \int_{0}^{t} g(s)^3 \frac{|x|^4\cG(2G(t)-G(s)+1,x) }{16(2G(t)-G(s)+1)^2}  |u (s,x)| dsdx
\\
&=: E_{41}(t)+E_{42}(t)+E_{43}(t).
\end{align*}
We estimate $E_{41}$. 
By the H\"{o}lder inequality and the Young inequality, we get
\begin{align*}
E_{41}(t)&\ceq (G(t)+1)^{n/2-2}  \int_{\R^n} \int_{0}^{t} g(s)^3 \cG(2G(t)-G(s)+1,x) |u (s,x)| dsdx
\\
&\leq  (G(t)+1)^{n/2-2} \l( \int_{\R^n} \int_{0}^{t} g(s)^{2p'+1} \cG(2G(t)-G(s)+1,x)  dsdx\r)^{1/p'}
\\
& \quad \times  \l(\int_{\R^n} \int_{0}^{t}  g(s)\cG(2G(t)-G(s)+1,x)  |u (s,x)|^p dsdx \r)^{1/p}
\\
& \ceq  ( G(t)+1)^{\frac{n}{2p'}-2} \l( \int_{0}^{t}  g(s)^{2p'+1} ds\r)^{1/p'} D(t)^{1/p}
\\
& \leq  \frac{1}{12} D(t) +C_{n,p} ( G(t)+1)^{\frac{n}{2}-2p'} \l( \int_{0}^{t}  g(s)^{2p'+1} ds\r).
\end{align*}
 By the same way, $E_{42}$ and $E_{43}$ can be estimated as follows. 
\begin{align*}
E_{42}(t),E_{43}(t) \leq  \frac{1}{12} D(t) +C_{n,p} ( G(t)+1)^{\frac{n}{2}-2p'} \l( \int_{0}^{t}  g(s)^{2p'+1} ds\r).
\end{align*}
By combining these estimates, we obtain
\begin{align*}
|E_{4}(t)| \leq  \frac{1}{4} D(t) +C_{n,p} ( G(t)+1)^{\frac{n}{2}-2p'} \l( \int_{0}^{t}  g(s)^{2p'+1} ds\r).
\end{align*}

\subsection{Reducing to an ordinary differential inequality}
\label{sec2.3}


We apply the above estimates to \eqref{eq3.5} to obtain an ordinary differential inequality. 

From \eqref{eq3.5}, by the estimate of $B$, 
we obtain
\begin{align*}
A(t)+g(t)A'(t)+C_{1} \frac{g(t)^2 F(t)}{G(t)+1} -E(t) \geq C(t) +D(t).
\end{align*}
By the estimate of $E$, we get
\begin{align*}
&A(t)+g(t)A'(t)+C_{2} \frac{g(t)^2 F(t)}{G(t)+1} 
\\
&+C_{3} ( G(t)+1)^{\frac{n}{2}-2p'} \l( \int_{0}^{t}  g(s)^{2p'+1} ds\r) 
\\
&+C_{4} (G(t)+1)^{\frac{n}{2}-p'} \l( \int_{0}^{t} g(s)|g'(s)|^{p'}ds\r)
\\
& \geq C(t)-E_2(t) +\frac{1}{2}D(t).
\end{align*}
Applying the estimate of $D$, we get, for any $t>0$,
\begin{align}
\label{eq3.6}
A(t)+g(t)A'(t)+C_{2} \frac{g(t)^2 F(t)}{G(t)+1}  
\geq 
C_5 \int_{0}^{t} \frac{g(s)F(s)^p}{(G(s)+1) ^{\frac{n}{2}(p-1)}}  ds +\eps H(t) -I(t),
\end{align}
where we set
\begin{align*}
H(t)&:= \eps^{-1}(C(t)-E_2(t) ),
\\
I(t)&:=C_{3} ( G(t)+1)^{\frac{n}{2}-2p'} \l( \int_{0}^{t}  g(s)^{2p'+1} ds\r) 
\\
&\quad +C_{4} (G(t)+1)^{\frac{n}{2}-p'} \l( \int_{0}^{t} g(s)|g'(s)|^{p'}ds\r).
\end{align*}
Multiplying $g(t)^{-1}e^{\Gamma(t)}=(e^{\Gamma(t)})'$, we get
\begin{align*}
&(e^{\Gamma(t)}A(t))'
+C_{2} \frac{g(t) e^{\Gamma(t)} F(t)}{G(t)+1}  
\\
& \geq C_5 (e^{\Gamma(t)})' \int_{0}^{t} \frac{g(s)F(s)^p}{(G(s)+1) ^{\frac{n}{2}(p-1)}}  ds 
+(e^{\Gamma(t)})'( \eps H(t) -I(t)).
\end{align*}
Integrating this on $[0,t]$, we obtain
\begin{align*}
&e^{\Gamma(t)}A(t)-A(0)+C_{2} \int_{0}^{t} \frac{g(\tau) e^{\Gamma(\tau)} F(\tau)}{G(\tau)+1}   d\tau
\\
&\geq 
C_5 \int_{0}^{t}  (e^{\Gamma(\tau)})' \int_{0}^{\tau} \frac{g(s)F(s)^p}{(G(s)+1) ^{\frac{n}{2}(p-1)}}  ds d\tau
+ \int_{0}^{t} (e^{\Gamma(\tau)})'( \eps H(\tau) -I(\tau))d\tau.
\end{align*}
Since we have the identity
\begin{align*}
\int_{0}^{t}  (e^{\Gamma(\tau)})' \int_{0}^{\tau} \frac{g(s)F(s)^p}{(G(s)+1) ^{\frac{n}{2}(p-1)}}  ds d\tau
= \int_{0}^{t} \frac{ (e^{\Gamma(t)}- e^{\Gamma(\tau)}) g(\tau)F(\tau)^p}{(G(\tau)+1) ^{\frac{n}{2}(p-1)}}  d\tau,
\end{align*}
and the estimate of $A$, we obtain
\begin{align*}
&e^{\Gamma(t)}F(t)+C_{2} \int_{0}^{t} \frac{g(\tau) e^{\Gamma(\tau)} F(\tau)}{G(\tau)+1}   d\tau
\\
&\geq 
C_5 \int_{0}^{t} \frac{ (e^{\Gamma(t)}- e^{\Gamma(\tau)}) g(\tau)F(\tau)^p}{(G(\tau)+1) ^{\frac{n}{2}(p-1)}}  d\tau
+ \int_{0}^{t} (e^{\Gamma(\tau)})'( \eps H(\tau) -I(\tau))d\tau+A(0).
\end{align*}
Multiplying $g^{-1}$ and integrating on $[0,t]$, we obtain
\begin{align}
\label{eq3.7}
& \int_{0}^{t} \frac{1}{g(s)} e^{\Gamma(s)}F(s) ds +C_{2}\int_{0}^{t} \frac{1}{g(s)}  \int_{0}^{s} \frac{g(\tau) e^{\Gamma(\tau)} F(\tau)}{G(\tau)+1}   d\tau ds
\\ \notag
&\geq 
C_5 \int_{0}^{t} \frac{1}{g(s)}  \int_{0}^{s} \frac{ (e^{\Gamma(s)}- e^{\Gamma(\tau)}) g(\tau)F(\tau)^p}{(G(\tau)+1) ^{\frac{n}{2}(p-1)}}  d\tau ds
\\ \notag
&+ \int_{0}^{t} \frac{1}{g(s)} \l\{ \int_{0}^{s} (e^{\Gamma(\tau)})'( \eps H(\tau) -I(\tau))d\tau+A(0)\r\} ds.
\end{align}
We consider the left hand side. 
\begin{align}
\label{eq3.8}
&\int_{0}^{t} \frac{1}{g(s)} e^{\Gamma(s)}F(s) ds +C_{2}\int_{0}^{t} \frac{1}{g(s)}  \int_{0}^{s} \frac{g(\tau) e^{\Gamma(\tau)} F(\tau)}{G(\tau)+1}   d\tau ds
\\ \notag
&\cleq \int_{0}^{t} \l\{ \frac{1}{g(s)} e^{\Gamma(s)}F(s) + \frac{1}{g(s)}  \int_{0}^{s} \frac{g(\tau) e^{\Gamma(\tau)} F(\tau)}{G(\tau)+1}  d\tau \r\}  ds.
\end{align}
Letting $\chi_\beta =1$ if $\beta=-1$ and $\chi_\beta=0$ if $\beta\in(-1,1)$, then we have 
\begin{align*}
&\frac{d}{ds} \l( \{ \log (s+1)+1\}^{\chi_\beta}(\Gamma(s)+1)\int_{0}^{s}\frac{g(\tau) e^{\Gamma(\tau)} F(\tau)}{G(\tau)+1}  d\tau\r)
\\
&=\frac{\chi_\beta}{s+1} (\Gamma(s)+1)\int_{0}^{s}\frac{g(\tau) e^{\Gamma(\tau)} F(\tau)}{G(\tau)+1}  d\tau
\\
&\quad +\{ \log (s+1)+1\}^{\chi_\beta}\frac{1}{g(s)}\int_{0}^{s}\frac{g(\tau) e^{\Gamma(\tau)} F(\tau)}{G(\tau)+1}  d\tau
\\
&\quad +\{ \log (s+1)+1\}^{\chi_\beta}(\Gamma(s)+1) \frac{g(s) e^{\Gamma(s)} F(s)}{G(s)+1}. 
\end{align*}
Since the first term in the right hand side is non-negative, $\{ \log (s+1)+1\}^{\chi_\beta} \geq1$ holds for  any $ s\geq 0$ and $\beta \in [-1,1)$, and we have, by Lemma \ref{lem2.2},
\begin{align*}
& \{ \log (s+1)+1\}^{\chi_\beta}(\Gamma(s)+1) \frac{g(s)}{G(s)+1}
\\
& \ceq 
\begin{cases}
\{ \log (s+1)+1\} (s+1)^{2} \frac{(s+1)^{-1}}{\log(s+1)+1}  & \text{ if } \beta=-1,
\\
(s+1)^{-\beta+1} \frac{(s+1)^{\beta}}{(s+1)^{\beta+1}} & \text{ if } \beta\in(-1,1)
\end{cases}
\\
& =
\begin{cases}
 s+1  & \text{ if } \beta=-1,
\\
(s+1)^{-\beta}  & \text{ if } \beta\in(-1,1)
\end{cases}
\\
& \ceq \frac{1}{g(s)},
\end{align*}
we get
\begin{align*}
&\frac{d}{ds} \l( \{ \log (s+1)+1\}^{\chi_\beta}(\Gamma(s)+1)\int_{0}^{s}\frac{g(\tau) e^{\Gamma(\tau)} F(\tau)}{G(\tau)+1}  d\tau\r)
\\
& \cgeq  \frac{1}{g(s)} e^{\Gamma(s)}F(s) + \frac{1}{g(s)}  \int_{0}^{s} \frac{g(\tau) e^{\Gamma(\tau)} F(\tau)}{G(\tau)+1}  d\tau.
\end{align*}
Therefore, by \eqref{eq3.8}, we get
\begin{align*}
&\int_{0}^{t} \frac{1}{g(s)} e^{\Gamma(s)}F(s) ds +C_{2}\int_{0}^{t} \frac{1}{g(s)}  \int_{0}^{s} \frac{g(\tau) e^{\Gamma(\tau)} F(\tau)}{G(\tau)+1}   d\tau ds
\\
&\cleq  \int_{0}^{t} \frac{d}{ds} \l( \{ \log (s+1)+1\}^{\chi_\beta}(\Gamma(s)+1)\int_{0}^{s}\frac{g(\tau) e^{\Gamma(\tau)} F(\tau)}{G(\tau)+1}  d\tau\r) ds
\\
&= \{ \log (t+1)+1\}^{\chi_\beta}(\Gamma(t)+1)\int_{0}^{t}\frac{g(\tau) e^{\Gamma(\tau)} F(\tau)}{G(\tau)+1}  d\tau
\\
&=:X(t).
\end{align*}
We set 
\begin{align*}
Y(t):= \int_{0}^{t} \frac{1}{g(s)}  \int_{0}^{s} \frac{ (e^{\Gamma(s)}- e^{\Gamma(\tau)}) g(\tau)F(\tau)^p}{(G(\tau)+1) ^{\frac{n}{2}(p-1)}}  d\tau ds.
\end{align*}
Then, by \eqref{eq3.7}, we get
\begin{align}
\label{eq3.9}
C_6 X(t) \geq C_5 Y(t)+  \int_{0}^{t} \frac{1}{g(s)} \l\{ \int_{0}^{s} (e^{\Gamma(\tau)})'( \eps H(\tau) -I(\tau))d\tau+A(0)\r\} ds.
\end{align}
Now, we have
\begin{align*}
Y'(t)= \frac{1}{g(t)}  \int_{0}^{t} \frac{ (e^{\Gamma(t)}- e^{\Gamma(\tau)}) g(\tau)F(\tau)^p}{(G(\tau)+1) ^{\frac{n}{2}(p-1)}}  d\tau,
\end{align*}
and 
\begin{align*}
Y''(t)
&=-  \frac{g'(t)}{g(t)^2}  \int_{0}^{t} \frac{ (e^{\Gamma(t)}- e^{\Gamma(\tau)}) g(\tau)F(\tau)^p}{(G(\tau)+1) ^{\frac{n}{2}(p-1)}}  d\tau 
 +   \frac{e^{\Gamma(t)} }{g(t)^2}  \int_{0}^{t} \frac{ g(\tau)F(\tau)^p}{(G(\tau)+1) ^{\frac{n}{2}(p-1)}}  d\tau 
\\
&=-\frac{g'(t)}{g(t)} Y'(t)+\frac{1}{g(t)}\l( Y'(t) + \frac{1}{g(t)}  \int_{0}^{t} \frac{ e^{\Gamma(\tau)} g(\tau)F(\tau)^p}{(G(\tau)+1) ^{\frac{n}{2}(p-1)}}  d\tau\r).
\end{align*}
Therefore, we have
\begin{align}
\label{eq2.11}
g(t)Y''(t)+g'(t)Y'(t) -Y'(t) =\frac{1}{g(t)}  \int_{0}^{t} \frac{ e^{\Gamma(\tau)} g(\tau)F(\tau)^p}{(G(\tau)+1)^{\frac{n}{2}(p-1)}}  d\tau.
\end{align}
On the other hand, by the H\"{o}lder inequality, we have
\begin{align}
\label{eq2.12}
X(t)&=\{ \log (t+1)+1\}^{\chi_\beta}(\Gamma(t)+1)\int_{0}^{t}\frac{g(\tau) e^{\Gamma(\tau)} F(\tau)}{G(\tau)+1}  d\tau
\\ \notag
&\leq \{ \log (t+1)+1\}^{\chi_\beta}(\Gamma(t)+1) \l( \int_{0}^{t}\frac{g(\tau) e^{\Gamma(\tau)} F(\tau)^p}{(G(\tau)+1) ^{\frac{n}{2}(p-1)}}  d\tau \r)^{1/p}
\\ \notag
&\quad \times \l( \int_{0}^{t}\frac{g(\tau) e^{\Gamma(\tau)}}{(G(\tau)+1) ^{\l(1-\frac{n(p-1)}{2p}\r)p'}}  d\tau\r)^{1/p'}
\\ \notag
&= \{ \log (t+1)+1\}^{\chi_\beta}(\Gamma(t)+1) \l( \int_{0}^{t}\frac{g(\tau) e^{\Gamma(\tau)} F(\tau)^p}{(G(\tau)+1) ^{\frac{n}{2}(p-1)}}  d\tau \r)^{1/p}
\\ \notag
&\quad \times \l( \int_{0}^{t}\frac{g(\tau) e^{\Gamma(\tau)}}{(G(\tau)+1) ^{p'-\frac{n}{2}}}  d\tau\r)^{1/p'}.
\end{align}
Therefore, combining \eqref{eq2.11} with \eqref{eq2.12}, we get
\begin{align}
\label{eq2.13}
&g(t)Y''(t)+g'(t)Y'(t) -Y'(t) 
\\ \notag
&=\frac{1}{g(t)}  \int_{0}^{t} \frac{ e^{\Gamma(\tau)} g(\tau)F(\tau)^p}{(G(\tau)+1)^{\frac{n}{2}(p-1)}}  d\tau
\\ \notag
&\geq  \frac{1}{g(t)}  X(t)^p \l[ \{ \log (t+1)+1\}^{\chi_\beta}(\Gamma(t)+1)\r]^{-p} \l( \int_{0}^{t}\frac{g(\tau) e^{\Gamma(\tau)}}{(G(\tau)+1) ^{p'-\frac{n}{2}}}  d\tau\r)^{-p/p'}.
\end{align}
We want to apply \eqref{eq3.9} to the above inequality. It is not clear whether the right hand side in \eqref{eq3.9} is positive for any $t>0$. However, for sufficiently large $t$, the right hand side is positive. To see this, we consider the second term in the right hand side of \eqref{eq3.9}, that is, 
\begin{align*}
 \int_{0}^{t} \frac{1}{g(s)} \l\{ \int_{0}^{s} (e^{\Gamma(\tau)})'( \eps H(\tau) -I(\tau))d\tau+A(0)\r\} ds.
\end{align*}
We recall that 
\begin{align*}
H(t)
&=\eps^{-1}(C(t)-E_2(t) )
\\
& =  \l( \frac{G(t)+1}{2G(t)+1}\r)^{n/2}  \int_{\R^n} e^{-\frac{|x|^2}{4(2G(t)+1)}} (u_0+g(0)u_1)(x)  dx 
\\
&\quad + \l( \frac{G(t)+1}{2G(t)+1}\r)^{n/2}  \frac{n g(0)^2}{2(2G(t)+1)}  \int_{\R^n}   e^{-\frac{|x|^2}{4(2G(t)+1)}} u_0 (x)dx
\\
&\quad - \l( \frac{G(t)+1}{2G(t)+1}\r)^{n/2}  \frac{g(0)^2 }{2G(t)+1} \int_{\R^n}  \frac{|x|^2}{4(2G(t)+1)} e^{-\frac{|x|^2}{4(2G(t)+1)}} u_0 (x)dx.
\end{align*}
Since $G(t) \to \infty $ as $t\to \infty$ and $u_0$, $u_1$ belong to $L^1(\R^n)$, 
the dominated convergence theorem implies
\begin{align*}
\lim_{t\to \infty} H(t) = 2^{-n/2}  \int_{\R^n} u_0(x)+g(0)u_1(x) dx=:J_0.  
\end{align*}

Moreover, by the assumption on the initial data, there exists $t_1 = t_1(n,u_0,u_1)>0$ such that 
\[ H(t) > \frac{J_0}{2}>0 \text{ for any } t \geq t_1.\]

Next, we treat $I$. 
\begin{align*}
I(t)&=C_{3} ( G(t)+1)^{\frac{n}{2}-2p'} \l( \int_{0}^{t}  g(s)^{2p'+1} ds\r) 
\\
&\quad +C_{4} (G(t)+1)^{\frac{n}{2}-p'} \l( \int_{0}^{t} g(s)|g'(s)|^{p'}ds\r)
\\
&=:I_1(t)+I_2(t).
\end{align*}
We consider the first term $I_1$. 
First, we consider the case of $\beta \in (-1,1)$. Since we have
\begin{align*}
\int_{0}^{t}  g(s)^{2p'+1} ds
\ceq \int_{0}^{t}  (s+1)^{\beta(2p'+1)} ds
 \cleq 
\begin{cases}
(t+1)^{\beta(2p'+1)+1} & \text{ if } \beta \neq -1/(2p'+1),
\\
\log(t+1) & \text{ if } \beta = -1/(2p'+1),
\end{cases}
\end{align*}
and Lemma \ref{lem2.2} (iv), we obtain
\begin{align*}
I_1(t)
& \cleq
\begin{cases}
(t+1)^{(\beta+1) \l( \frac{n}{2}+1\r)-2p'} &  \text{ if } \beta \neq -1/(2p'+1),
\\
(t+1)^{(\beta+1)\l( \frac{n}{2}-2p'\r)} \log(t+1) & \text{ if } \beta = -1/(2p'+1)
\end{cases} 
\\
&\leq 
\begin{cases}
(t+1)^{(\beta+1) \l( \frac{n}{2}+1\r)-2p'} &  \text{ if } \beta \neq -1/(2p'+1),
\\
(t+1)^{(\beta+1)\l( \frac{n}{2}-2p'\r)+\delta} & \text{ if } \beta = -1/(2p'+1),
\end{cases} 
\end{align*}
for any $\delta>0$.
Since $p=1+2/n$, we have $p' = 1+n/2$. Therefore, in the first case, \textit{i.e.} $ \beta \neq -1/(2p'+1)$,  we get
\begin{align*}
(\beta+1) \l( \frac{n}{2}+1\r)-2p' 
=(\beta-1) \l( \frac{n}{2}+1\r)
<0
\end{align*}
since $\beta<1$. In the second case, \textit{i.e.} $ \beta = -1/(2p'+1)$, noting that $\beta = -1/(2p'+1)>-1$,  we have
\begin{align*}
(\beta+1) \l( \frac{n}{2} -2p'\r) 
&\leq -(\beta+1) \l( \frac{n}{2} +2\r)<0.
\end{align*}
Take sufficiently small $\delta>0$ such that $(\beta+1) \l( n/2 -2p'\r) +\delta<0$. Then, we take $t_2=t_2(\eps_0, \beta,n,p)>0$ such that
\begin{align*}
&t_2 =( C \eps_0)^{\frac{1}{(\beta+1) \l( \frac{n}{2}+1\r)-2p'}} -1 \text{ if }  \beta \neq -1/(2p'+1),
\\
&t_2 =( C \eps_0)^{\frac{1}{(\beta+1)\l( \frac{n}{2}-2p'\r)+\delta}}-1 \text{ if } \beta = -1/(2p'+1).
\end{align*}
Then we have 
\begin{align*}
&(t+1)^{(\beta+1) \l( \frac{n}{2}+1\r)-2p'} \leq C \eps_0 \text{ if }  \beta \neq -1/(2p'+1),
\\
&(t+1)^{(\beta+1)\l( \frac{n}{2}-2p'\r)+\delta} \leq C \eps_0 \text{ if } \beta = -1/(2p'+1),
\end{align*}
for any $t>t_{2}$ and thus we have $I_1(t)\leq C \eps_0$ for $t>t_2$. 

Next, we consider the case of $\beta=-1$. Then, $\beta \neq  -1/(2p'+1)$. 
Therefore, 
\begin{align*}
I_1(t)
& \cleq (\log (t+1)+1)^{\frac{n}{2}-2p'} (t+1)^{\beta(2p'+1)+1} 
\\
& \leq (\log (t+1)+1)^{\frac{n}{2}-2p'} (t+1)^{-2p'}
\\
& \leq (t+1)^{-2p'},
\end{align*}
since $n/2-2p'<0$. Thus, $I_1(t)\cleq \eps_0$ for $t>t_2=t_2(\eps_0):= C \eps_0^{-1/(2p')}-1$.

Secondly, we treat the second term $I_2$. By Lemma \ref{lem2.2} (iii), there exists $t_3=t_3(\beta)>0$ such that $|g'(t)|\ceq (t+1)^{\beta-1}$ for $t>t_3$. 
We take $t_4=t_4 (\eps_0, \beta,n,p)>0$ such that 
\begin{align*}
t_4= \l(\frac{1}{2} \eps_0\r)^{\frac{1}{(\beta-1)p'}}-1. 
\end{align*}
Taking sufficiently small $\eps_0>0$, we may assume that $t_4>t_3$. 
Then, $|g'(t)| \cleq (t+1)^{\beta-1} \leq (\eps_0/2)^{1/p'}$ for $t>t_4$. 
As seen in Lemma \ref{lem2.2} (i), $g'(t)$ converges to $0$ as $t \to \infty$ . Thus, we can define $m:=\max_{t\in[0,\infty)} |g'(t)| < \infty$. 
Then, for $t>t_4$, we have 
\begin{align*}
\int_{0}^{t} g(s)|g'(s)|^{p'}ds
&=  \int_{0}^{t_4} g(s)|g'(s)|^{p'}ds + \int_{t_4}^{t} g(s)|g'(s)|^{p'}ds
\\
&\leq m^{p'} G(t_4)+ \frac{C}{2} \eps_0 G(t),
\end{align*}
where we have used $|g'(t)|^{p'}<\eps_0/2 $ for $t>t_4$
We show that there exists $t_5=t_5 (\eps_0,\beta,n,p)>0$ such that $m^{p'} G(t_4) \cleq \eps_0 G(t)/2$. We consider the case of $\beta \neq -1$. Then, we have
\begin{align*}
G(t_4 ) \leq C (t_4+1)^{\beta+1}  \ceq  \l( \frac{1}{2} \eps_0\r)^{\frac{\beta+1}{(\beta-1)p'}}.
\end{align*}
Take $t_5=t_5(\eps_0,\beta,n,p) :=  \eps_0^{\frac{1}{(\beta-1)p'}-\frac{1}{\beta+1}}-1$. Then, for $t>t_5$,
\begin{align*}
\eps_0 G(t) \geq \eps_0 G(t_5) \cgeq \eps_0^{\frac{\beta+1}{(\beta-1)p'}} \geq G(t_4).
\end{align*}
We consider the case of $\beta = -1$. Then, we have
\begin{align*}
G(t_4) \leq C \log(t_4+1)  \leq C \log\l\{  \l(\frac{1}{2} \eps_0\r)^\frac{1}{(\beta-1)p'} \r\}.
\end{align*}
Take $t_5= t_5 (\eps_0,\beta,n,p) = \exp \{ \eps_0^{-1}\log (\eps_0/2)^{\frac{1}{(\beta-1)p'} }\}-1$. Then, for $t>t_5$,
\begin{align*}
\eps_0 G(t) \geq \eps_0 G(t_5) \cgeq \eps_0 \log(  t_5+1)= \log\l\{ \l(\frac{1}{2} \eps_0\r)^\frac{1}{(\beta-1)p'} \r\}  \cgeq G(t_4).
\end{align*}
Therefore, for $t>t_5$, we have
\begin{align*}
\int_{0}^{t} g(s)|g'(s)|^{p'}ds
\leq m^{p'} G(t_4)+ \frac{C}{2} \eps_0 G(t)
\leq C \eps_0 G(t),
\end{align*}
and thus we obtain
\begin{align*}
I_2(t) 
\cleq  (G(t)+1)^{\frac{n}{2}-p'+1} \frac{\int_{0}^{t} g(s)|g'(s)|^{p'}ds}{G(t)}
\cleq  \eps_0,
\end{align*}
where we also use $n/2-p'+1 = 0$ in the last inequality. 

By the above argument, we get 
\begin{align*}
I(t) \leq C' \eps_{0}
\end{align*}
for $t>\max\{t_2,t_5\}$. 
Let $\eps_0$ be sufficiently small such that $C' \eps_0< J_0 \eps /4$. 
Take large $t_6=t_6(\eps, \beta, n,p)=\max \{t_1, t_2,\cdots, t_5 \}$, where $t_j$ ($j=1,\cdots, 5$) are defined before. Noting that $A(0)=\eps A_0$, where $A_0$ is a constant,  then, we get 
\begin{align*}
 &\int_{0}^{t} \frac{1}{g(s)} \l\{ \int_{0}^{s} (e^{\Gamma(\tau)})'( \eps H(\tau) -I(\tau))d\tau+A(0)\r\} ds
 \\
& = \int_{0}^{t_6} \frac{1}{g(s)}  \int_{0}^{s}  (e^{\Gamma(\tau)})'( \eps H(\tau) -I(\tau))d\tau  ds
 \\
&\quad + \int_{t_6}^{t} \frac{1}{g(s)}  \int_{0}^{t_6} (e^{\Gamma(\tau)})'( \eps H(\tau) -I(\tau))d\tau  ds
 \\
&\quad + \int_{t_6}^{t} \frac{1}{g(s)}   \int_{t_6}^{s} (e^{\Gamma(\tau)})'( \eps H(\tau) -I(\tau))d\tau  ds
 \\
&\quad +  \eps A_0 \Gamma (t)
\\
& =: I + I\!\!I + I\!\!I\!\!I + \eps A_0 \Gamma (t)
\end{align*}
We estimate the terms $I$,  $I\!\!I$, and $I\!\!I\!\!I$. 
For $t>t_6$, we have $H(t)>J_0/2>0$ as seen before. Moreover, we have $I(t) \leq C \eps_0 \leq J_0 \eps /4 $ for $t>t_6$. This implies that $\eps H(t) -I(t) \geq J_0 \eps /4$ for $t>t_6$. Noting that $(e^{\Gamma(t)})' \geq0$, we can estimate $I\!\!I\!\!I$ as follows. 
\begin{align*}
I\!\!I\!\!I
&\geq  \frac{J_0}{4} \eps    \int_{t_6}^{t} \frac{1}{g(s)}  \int_{t_6}^{s}  (e^{\Gamma(\tau)})' d\tau  ds
\\
&=  \frac{J_0}{4} \eps   \int_{t_6}^{t} \frac{1}{g(s)}  (e^{\Gamma(s)}-e^{\Gamma(t_6)}) ds
\\
&=  \frac{J_0}{4} \eps \l\{ (e^{\Gamma(t)}-e^{\Gamma(t_6)}) - (\Gamma(t)-\Gamma(t_6) ) e^{\Gamma(t_6)}\r\}.
\end{align*}
Next, we consider the estimate of $I\!\!I$. Now, since $H$ is continuous and has a limit as $t \to \infty$, $H$ is bounded on $[0,\infty)$. Therefore, we can set $M_h:=\max_{t\in [0,\infty)}H(t)<\infty$. Similarly, we can also define $M_i:=\max_{t\in [0,\infty)} I(t)<\infty$. Then, since we have $\eps H(t) -I(t) \geq - (\eps M_h +M_i)$, $I\!\!I$ can be estimated as follows. 
\begin{align*}
 I\!\!I 
 & \geq -  (\eps M_h +M_i)\int_{t_6}^{t} \frac{1}{g(s)}  \int_{0}^{t_6}  (e^{\Gamma(\tau)})'  d\tau  ds
 \\
 & = -(\eps M_h +M_i)  (\Gamma(t)-\Gamma(t_6)) (e^{\Gamma(t_{\eps})}-1).
\end{align*}
At last, in the same manner, $I$ can be estimated as follows. 
\begin{align*}
I
&\geq  - (\eps M_h +M_i) \int_{0}^{t_6} \frac{1}{g(s)} \int_{0}^{s}   (e^{\Gamma(\tau)})' d\tau  ds
\\
&= - (\eps M_h +M_i) \l\{ (e^{\Gamma(t_6)}-1) - \Gamma(t_6)  \r\}.
\end{align*}
Combining these estimates, we obtain
\begin{align*}
I + I\!\!I + I\!\!I\!\!I + \eps A_0 \Gamma (t)
&\geq - (\eps M_h +M_i) \l\{ e^{\Gamma(t_6)}-1 +(\Gamma(t)-\Gamma(t_6)) e^{\Gamma(t_6)} -\Gamma(t)  \r\}
\\
& \quad + C\eps \l\{ (e^{\Gamma(t)}-e^{\Gamma(t_6)}) - (\Gamma(t)-\Gamma(t_6) ) e^{\Gamma(t_6)}\r\} + \eps A_0 \Gamma (t)
\\
&\geq - (\eps M_h +M_i) \l\{ e^{\Gamma(t_6)} +(\Gamma(t)-\Gamma(t_6)) e^{\Gamma(t_6)}  \r\}
\\
&\quad + C\eps \l\{ (e^{\Gamma(t)}-e^{\Gamma(t_6)}) - (\Gamma(t)-\Gamma(t_6) ) e^{\Gamma(t_6)}\r\} + \eps A_0 \Gamma (t) 
\\
&= - (\eps M_h +M_i) \l\{ e^{\Gamma(t_6)} +(\Gamma(t)-\Gamma(t_6)) e^{\Gamma(t_6)}  \r\}
\\
& \quad + C\eps \l\{ (e^{\Gamma(t)}-e^{\Gamma(t_6)}) - (\Gamma(t)-\Gamma(t_6) ) e^{\Gamma(t_6)}\r\}  
\\
& \quad +  \eps A_0 \Gamma (t)+\Gamma(t_6)\eps |A_0| e^{\Gamma(t_6)}-\Gamma(t_6)\eps |A_0| e^{\Gamma(t_6)}
\\
&\geq - (\eps M_h +M_i) \l\{ e^{\Gamma(t_6)} +(\Gamma(t)-\Gamma(t_6)) e^{\Gamma(t_6)}  \r\}
\\
&\quad + C\eps \l\{ (e^{\Gamma(t)}-e^{\Gamma(t_6)}) - (\Gamma(t)-\Gamma(t_6) ) e^{\Gamma(t_6)}\r\}  
\\
& \quad - \Gamma (t) \eps |A_0| e^{\Gamma(t_6)} +\Gamma(t_6)\eps |A_0| e^{\Gamma(t_6)}-\Gamma(t_6)\eps |A_0| e^{\Gamma(t_6)}
\\
&= - (\eps M_h +M_i) (\Gamma(t)-\Gamma(t_6)) e^{\Gamma(t_6)}
- (\eps M_h +M_i) e^{\Gamma(t_6)} 
\\
&\quad + C\eps e^{\Gamma(t)}- C\eps e^{\Gamma(t_6)} - C\eps  (\Gamma(t)-\Gamma(t_6) ) e^{\Gamma(t_6)}
\\
&\quad - ( \Gamma (t) - \Gamma(t_6))\eps |A_0| e^{\Gamma(t_6)} -\Gamma(t_6)\eps |A_0| e^{\Gamma(t_6)}
\\
&\geq - (\eps M_h +M_i+C\eps+\Gamma(t_6) \eps |A_0| ) (\Gamma(t)-\Gamma(t_6)) e^{\Gamma(t_6)}
\\
&\quad - (\eps M_h +M_i +C\eps+\Gamma(t_6) \eps |A_0|) e^{\Gamma(t_6)} 
+ C\eps e^{\Gamma(t)}
\\
&= e^{\Gamma(t_6)}  \l\{  C\eps e^{\Gamma(t)-\Gamma(t_6)} -K_{\eps} -K_{\eps} (\Gamma(t)-\Gamma(t_6))\r\},
\end{align*}
where $K_{\eps}:=\eps M_h +M_i+C\eps+\Gamma(t_6)\eps |A_0|$.
We set $\lambda(t):=\Gamma(t)-\Gamma(t_6)$ and $U(t):= C\eps e^{\lambda(t)} -K_{\eps}-K_{\eps}\lambda(t)$. We find $t$ such that $U(t)> C_7 \eps e^{\lambda(t)}$, where $C_7<C$. 
If we take $t_7$ such that $\lambda(t_7)> 2\log ((C-C_7)^{-1}K_{\eps} \eps^{-1}) + M$, where $M$ is a positive constant such that $e^{s/2}<e^{s}/(s+1)$ for $s>M$, then $U(t)-C_7 \eps e^{\lambda(t)}$ is increasing for $t>t_7$ and thus
\begin{align*}
U(t)-C_7 \eps e^{\lambda(t)}
&= (C-C_7)\eps e^{\lambda(t)} -K_{\eps}-K_{\eps}\lambda(t)
\\
&\geq (C-C_7)\eps e^{\lambda(t_7)} -K_{\eps}-K_{\eps}\lambda(t_7)
\\
&=\l\{ (C-C_7)\eps e^{\lambda(t_7)} \{\lambda(t_7)+1\}^{-1} -K_{\eps} \r\} (\lambda(t_7)+1)
\\
&> \l\{ (C-C_7)\eps e^{\lambda(t_7)/2} -K_{\eps} \r\}( \lambda(t_7)+1)
\\
&\geq \l\{ (C-C_7)\eps e^{\log ((C-C_7)^{-1}K_{\eps} \eps^{-1})} -K_{\eps} \r\} ( \lambda(t_7)+1)
=0.
\end{align*}
Now, since we have
\begin{align*}
\lambda(t_7) = \Gamma(t_7)-\Gamma(t_6) \cgeq \{ (t_7+1)^{-\beta+1} - (t_6+1)^{-\beta+1}\},
\end{align*}
it is enough to take $t_7$ such that
\begin{align*}
t_7 =   \l[(t_6+1)^{-\beta+1} + C\l\{ 2\log ((C-C_7)^{-1}K_{\eps} \eps^{-1}) + M\r\} \r]^{1/(-\beta+1)}-1. 
\end{align*} 

Take $t_{\eps}:= \max\{t_6,t_7\}$. Then, by the above argument, we obtain
\begin{align*}
\int_{0}^{t} \frac{1}{g(s)} \l\{ \int_{0}^{s} (e^{\Gamma(\tau)})'( \eps H(\tau) -I(\tau))d\tau+A(0)\r\} ds
\geq C_7 \eps e^{\Gamma(t)},
\end{align*}
for $t>t_{\eps}$ and thus, by \eqref{eq3.9}, we have
\begin{align}
\label{eq2.14}
C_6 X(t) 
\geq C_5 Y(t)+  C_7 \eps e^{\Gamma(t)} \text{ for any } t>t_{\eps}.
\end{align}
Therefore, applying \eqref{eq2.14} to \eqref{eq2.13} for $t>t_{\eps}$, we obtain 
\begin{align}
\label{eq2.15}
&g(t)^2Y''(t)+g(t)g'(t)Y'(t) -g(t) Y'(t) 
\\ \notag
&\geq  C_5 X(t)^p \l[ \{ \log (t+1)+1\}^{\chi_\beta}(\Gamma(t)+1)\r]^{-p} \l( \int_{0}^{t}\frac{g(\tau) e^{\Gamma(\tau)}}{(G(\tau)+1) ^{p'-\frac{n}{2}}}  d\tau\r)^{-p/p'}
\\ \notag
& \cgeq (Y(t)+ \eps e^{\Gamma(t)})^p \l[ \{ \log (t+1)+1\}^{\chi_\beta}(\Gamma(t)+1)\r]^{-p} \l( \int_{0}^{t}\frac{g(\tau) e^{\Gamma(\tau)}}{(G(\tau)+1) ^{p'-\frac{n}{2}}}  d\tau\r)^{-p/p'}.
\end{align}
We set $Z(t):=e^{-\Gamma(t)} Y(t)$. Then, simple calculations give 
\begin{align*}
Y'(t) &= g(t)^{-1} e^{\Gamma(t)} Z(t)+ e^{\Gamma(t)} Z'(t),
\\
Y''(t) & = -g(t)^{-2} g'(t) e^{\Gamma(t)} Z(t) +g(t)^{-2}  e^{\Gamma(t)} Z(t) + 2 g(t)^{-1} e^{\Gamma(t)} Z' (t)+ e^{\Gamma(t)} Z''(t),
\end{align*}
and thus we have
\begin{align*}
&g(t)^2Y''(t)+g(t)g'(t)Y'(t) -g(t) Y'(t) 
\\
& =g(t)^2  e^{\Gamma(t)} Z''(t) +g(t)(g'(t)+1) e^{\Gamma(t)} Z'(t).
\end{align*}
Therefore, from \eqref{eq2.15}, we obtain
\begin{align*}
&g(t)^2  e^{\Gamma(t)} Z''(t) +g(t)(g'(t)+1) e^{\Gamma(t)} Z'(t)
\\
& \cgeq  e^{p\Gamma(t)} (Z(t)+ \eps)^p \l[ \{ \log (t+1)+1\}^{\chi_\beta}(\Gamma(t)+1)\r]^{-p} \l( \int_{0}^{t}\frac{g(\tau) e^{\Gamma(\tau)}}{(G(\tau)+1) ^{p'-\frac{n}{2}}}  d\tau\r)^{-p/p'}.
\end{align*}
Define $W(t):=Z(t)+\eps$. Then,  noting that $g'+1=gb$, we get
\begin{align}
\label{eq2.16}
& W''(t) +b(t) W'(t)
\\ \notag
& \cgeq W(t)^p g(t)^{-2} e^{(p-1)\Gamma(t)}  \l[ \{ \log (t+1)+1\}^{\chi_\beta}(\Gamma(t)+1)\r]^{-p} \l( \int_{0}^{t}\frac{g(\tau) e^{\Gamma(\tau)}}{(G(\tau)+1) ^{p'-\frac{n}{2}}}  d\tau\r)^{-p/p'}.
\end{align}

\begin{lemma}
\label{lem2.6}
For any $t>0$, we have
\begin{align*}
\int_{0}^{t}\frac{g(\tau) e^{\Gamma(\tau)}}{(G(\tau)+1) ^{p'-\frac{n}{2}}}  d\tau \cleq \frac{g(t)^2 e^{\Gamma(t)}}{(G(t)+1) ^{p'-\frac{n}{2}}}.
\end{align*}
\end{lemma}

\begin{proof}
It is enough to consider the case of $t>\widetilde{t}$, where $\widetilde{t}>0$ is defined later. 
We have
\begin{align*}
\int_{\widetilde{t}}^{t}\frac{g(\tau) e^{\Gamma(\tau)}}{(G(\tau)+1) ^{p'-\frac{n}{2}}}  d\tau 
&=  \int_{\widetilde{t}}^{t}\frac{ e^{ \frac{1}{2}\Gamma(\tau) }}{\frac{1}{2}g(\tau)}
\cdot \frac{\frac{1}{2} g(\tau)^2 e^{\frac{1}{2} \Gamma(\tau)}}{(G(\tau)+1) ^{p'-\frac{n}{2}}} d\tau.
\end{align*}
We set 
\begin{align*}
f(\tau):= \frac{\frac{1}{2} g(\tau)^2 e^{\frac{1}{2} \Gamma(\tau)}}{(G(\tau)+1) ^{p'-\frac{n}{2}}}.
\end{align*}
Then, by a simple calculation, we have
\begin{align*}
f'(\tau)= \frac{e^{\frac{1}{2}\Gamma(\tau)} g(\tau) (G(\tau)+1)^{p'-\frac{n}{2}-1} \l\{\l(g'(\tau) +\frac{1}{4}\r)(G(\tau)+1)- \frac{1}{2}\l(p'-\frac{n}{2}\r)g(\tau)^2 \r\} }{ (G(\tau)+1) ^{2p'-n}}.
\end{align*}
By Lemma \ref{lem2.2} (i), $|g'(\tau)|<1/8$ for sufficiently large $\tau>0$. Thus, we have
\begin{align*}
&\l(g'(\tau) +\frac{1}{4}\r)(G(\tau)+1)- \frac{1}{2}\l(p'-\frac{n}{2}\r)g(\tau)^2 
\\
& \geq \l(\frac{1}{4}-\frac{1}{8}\r)(G(\tau)+1)- \frac{1}{2}\l(p'-\frac{n}{2}\r)g(\tau)^2 
\\
& \cgeq
\l\{
\begin{array}{ll}
\displaystyle \frac{1}{8}(\tau+1)^{\beta+1} - \frac{C}{2}\l(p'-\frac{n}{2}\r)(\tau+1)^{2\beta}, & \text{ if }\beta \in (-1,1),
\\
\ 
\\
\displaystyle \frac{1}{8}\log (\tau+1)- \frac{C}{2}\l(p'-\frac{n}{2}\r)(\tau+1)^{-2}, & \text{ if }\beta=-1,
\end{array}
\r.
\\
&>0,
\end{align*}
for large $\tau>0$ since $\beta<1$. 
We take $\widetilde{t}>0$ such that $f'(\tau)>0$ for $\tau>\widetilde{t}$. Then, $f(\tau)$ is increasing for $\tau>\widetilde{t}$. 
Therefore, we obtain
\begin{align*}
\int_{\widetilde{t}}^{t}\frac{g(\tau) e^{\Gamma(\tau)}}{(G(\tau)+1) ^{p'-\frac{n}{2}}}  d\tau 
&\leq  \int_{\widetilde{t}}^{t}\frac{ e^{ \frac{1}{2}\Gamma(\tau) }}{\frac{1}{2}g(\tau)}  d\tau
\cdot \frac{\frac{1}{2} g(t)^2 e^{\frac{1}{2} \Gamma(t)}}{(G(t)+1) ^{p'-\frac{n}{2}}}
\\
& \leq  e^{\frac{1}{2}\Gamma(t)}\cdot \frac{\frac{1}{2} g(t)^2 e^{\frac{1}{2} \Gamma(t)}}{(G(t)+1) ^{p'-\frac{n}{2}}}
\\
&= \frac{ g(t)^2 e^{ \Gamma(t)}}{2(G(t)+1) ^{p'-\frac{n}{2}}},
\end{align*}
for $t>\widetilde{t}$. On the other hand, we have
\begin{align*}
\int_{0}^{\widetilde{t}} \frac{g(\tau) e^{\Gamma(\tau)}}{(G(\tau)+1) ^{p'-\frac{n}{2}}}  d\tau 
&\leq C.
\end{align*}
Combining them, we get the statement. 
\end{proof}

Then, by Lemma \ref{lem2.6} and \eqref{eq2.16}, we get
\begin{align*}
& W''(t) +b(t) W'(t)
\\
& \cgeq W(t)^p g(t)^{-2} e^{(p-1)\Gamma(t)}  \l[ \{ \log (t+1)+1\}^{\chi_\beta}(\Gamma(t)+1)\r]^{-p} \l( \frac{g(t)^2 e^{\Gamma(t)}}{(G(t)+1) ^{p'-\frac{n}{2}}}\r)^{-p/p'}
\\
&= W(t)^p g(t)^{-2p}  \l[ \{ \log (t+1)+1\}^{\chi_\beta}(\Gamma(t)+1)\r]^{-p} (G(t)+1) ^{\l(p'-\frac{n}{2}\r)(p-1)},
\end{align*}
for $t\geq t_{\eps}$. 
To derive an ordinary differential inequality, we consider two cases, \textit{i.e.} $\beta \in (-1,1)$ and $\beta=-1$, separately.

{\bf Case1.} We consider the case of $\beta \in (-1,1)$. 
When $\beta\in (-1,1)$, by Lemma \ref{lem2.2}, we have
\begin{align*}
&g(t)^{-2p}  \l[ \{ \log (t+1)+1\}^{\chi_\beta}(\Gamma(t)+1)\r]^{-p} (G(t)+1) ^{\l(p'-\frac{n}{2}\r)(p-1)}
\\
&= g(t)^{-2p}  (\Gamma(t)+1)^{-p} (G(t)+1) ^{\l(p'-\frac{n}{2}\r)(p-1)}
\\
&\ceq (t+1)^{-2\beta p} (t+1)^{(\beta-1) p} (t+1) ^{\l(p'-\frac{n}{2}\r)(p-1)(\beta+1)}
\\
& =  (t+1) ^{\l\{ \l(p'-\frac{n}{2}\r)(p-1) -p \r\}(\beta+1)}.
\end{align*} 
Since $p'=p/(p-1)$,  we have
\begin{align*}
\l\{ \l(p'-\frac{n}{2}\r)(p-1) -p \r\}(\beta+1)
=-\frac{n}{2}(p-1)(\beta+1).
\end{align*}
Thus, we obtain
\begin{align*}
W''(t) +b(t) W'(t) \cgeq  W(t)^p (t+1) ^{-\frac{n}{2}(p-1)(\beta+1)}.
\end{align*}
for $t\geq t_{\eps}$. 
Now, by the Fubini--Tonelli theorem, we have 
\begin{align*}
W'(t)
&=Z'(t)
\\
&=g(t)^{-1} e^{-\Gamma(t)}\l\{
 \int_{0}^{t} \frac{ (e^{\Gamma(t)}- e^{\Gamma(\tau)}) g(\tau)F(\tau)^p}{(G(\tau)+1) ^{\frac{n}{2}(p-1)}}  d\tau \r.
\\
&\quad  \l.
 -  \int_{0}^{t} \frac{1}{g(s)}  \int_{0}^{s} \frac{ (e^{\Gamma(s)}- e^{\Gamma(\tau)}) g(\tau)F(\tau)^p}{(G(\tau)+1) ^{\frac{n}{2}(p-1)}}  d\tau ds\r\} 
 \\
 &=g(t)^{-1} e^{-\Gamma(t)}
 \int_{0}^{t} \frac{1}{g(s)}  \int_{0}^{s} \frac{  e^{\Gamma(\tau)} g(\tau)F(\tau)^p}{(G(\tau)+1) ^{\frac{n}{2}(p-1)}}  d\tau ds
\\
& \geq 0,
\end{align*}
for any $t>0$. 
Therefore, since $W'$ is positive and $p=1+2/n$, we get
\begin{align}
\label{eq2.17}
(t+1)^{\beta} W''(t) +C W'(t) \cgeq \frac{ W(t)^p }{ t+1} \text{ for }t\geq t_{\eps}.
\end{align}

{\bf Case2.} We consider the case of $\beta=-1$. 
When $\beta=-1$, by Lemma \ref{lem2.2}, we have
\begin{align*}
&g(t)^{-2p}  \l[ \{ \log (t+1)+1\}^{\chi_\beta}(\Gamma(t)+1)\r]^{-p} (G(t)+1) ^{\l(p'-\frac{n}{2}\r)(p-1)}
\\
&\ceq (t+1)^{2p} \{ \log (t+1)+1\} ^{-p} (t+1)^{-2p} (\log(t+1)+1) ^{\l(p'-\frac{n}{2}\r)(p-1)}
\\
& = (\log(t+1)+1) ^{\l(p'-\frac{n}{2}\r)(p-1)-p}.
\end{align*} 
Since $p'=p/(p-1)$,  we have
\begin{align*}
\l(p'-\frac{n}{2}\r)(p-1)-p
=-\frac{n}{2}(p-1).
\end{align*}
Therefore, we obtain
\begin{align*}
W''(t) +b(t) W'(t) \cgeq  W(t)^p (\log(t+1)+1) ^{-\frac{n}{2}(p-1)}
\end{align*}
for $t\geq t_{\eps}$. 
Since $W'$ is positive and $p=1+2/n$, we get
\begin{align}
\label{eq2.18}
(t+1)^{-1} W''(t) +C W'(t) \cgeq \frac{ W(t)^p }{ (t+1)(\log(t+1)+1) } \text{ for }t\geq t_{\eps}.
\end{align}

We get objective ordinary inequalities.

\begin{proof}[Proof of Theorem \ref{thm1.1}]
We prove 
\begin{align}
\label{eq2.19}
\l\{
\begin{array}{ll}
T(\eps) \leq t_\eps + \exp(C\eps ^{-(p-1)}), & \text{ when } \beta \in (-1,1),
\\
T(\eps) \leq t_\eps + \exp \l(\exp(C\eps ^{-(p-1)})\r), & \text{ when } \beta =-1.
\end{array}
\r.
\end{align}
Since the above estimates hold if $T(\eps)<t_\eps$, we may assume that $T(\eps)>t_\eps$. Then, by the above argument, \eqref{eq2.17} or \eqref{eq2.18} hold for $t \in (t_\eps, T(\eps))$. 
To apply Lemmas \ref{lemA.1} and \ref{lemA.2}  in Appendix \ref{secA}  to \eqref{eq2.17} and \eqref{eq2.18} respectively, we check the positivity condition on the initial data. We have
\begin{align*}
W(t_\eps)&=Z(t_\eps)+\eps>0.
\end{align*}
Moreover, as seen above, we also have 
\begin{align*}
W'(t_\eps) \geq 0.
\end{align*}
Taking sufficiently small $\eps>0$, we can use Lemmas \ref{lemA.1} and \ref{lemA.2} and thus we get \eqref{eq2.19}. 
By the definition of $t_\eps$, for sufficiently small $\eps>0$, we have
$t_\eps <  \exp(C\eps ^{-(p-1)})$ when $\beta \in (-1,1)$ and $t_\eps <  \exp \l(\exp(C\eps ^{-(p-1)})\r)$ when $\beta=-1$. Therefore, we obtain the sharp upper estimate of the lifespan:
\begin{align*}
\l\{
\begin{array}{ll}
T(\eps) \leq  \exp(C\eps ^{-(p-1)}), & \text{ when } \beta \in (-1,1),
\\
T(\eps) \leq  \exp \l(\exp(C\eps ^{-(p-1)})\r), & \text{ when } \beta =-1.
\end{array}
\r.
\end{align*}
\end{proof}


\appendix


\section{Lemmas}
\label{secA}

We give the estimates of the lifespan for the ordinary differential inequalities.

\begin{lemma}
\label{lemA.1}
Let $f$ satisfy the following inequality.
\begin{align}
\l\{
\begin{array}{ll}
(t+1)^{\beta} f''(t) + C_1 f'(t) \geq C_2 \frac{ f(t)^p }{ t+1},  & \text{for } t_0\leq  t <T,
\\
f(t_0)>0,
\\
f'(t_0) \geq 0,
\end{array}
\r.
\end{align}
where $C_1\geq 1$ and $C_2>0$ are constants, $p>1$, and $t_0\geq 0$ satisfies $C_1-e^{(\beta-1)\log(t_0+1)} >0$. Then, we have
\begin{align*}
T \leq  \exp (C \eps^{-(p-1)}).
\end{align*}
\end{lemma}

\begin{proof}
Substituting $e^{t}-1$ into $t$ in the ordinary differential inequality, we obtain
\begin{align}
\label{eqA.2}
e^{\beta t} f''(e^{t}-1) +C_1 f'(e^{t}-1) \geq C_2 \frac{f(e^{t}-1)^p}{e^{t}}
\end{align}
for any $t \geq \log(t_0+1)$. 
We set $h(t):=f(e^{t}-1)$. Then, we have
\begin{align*}
h'(t)&=f'(e^{t}-1)e^{t},
\\
h''(t)&=f''(e^{t}-1)e^{2t} +h'(t).
\end{align*}
Therefore, by \eqref{eqA.2}, we have
\begin{align*}
e^{(\beta -1)t} h''(t)  +(C_1 -e^{(\beta -1)t}) h'(t)  \geq C_2 h(t)^p
\end{align*}
for any $t \geq \log(t_0+1)$. Since 
\begin{align*}
0 &< f(t_0) = h( \log(t_0+1)),
\\
0 & \leq f'(t_0) = h'( \log(t_0+1))(t_0+1)^{-1}
\end{align*}
mean that $h(\tau_0)>0$ and $h'(\tau_0) \geq0$, where $\tau_0:= \log(t_0+1)$, we obtain
\begin{align*}
\l\{
\begin{array}{ll}
e^{ (\beta -1)\tau } h''(\tau)  +(C_1 -e^{(\beta -1)\tau}) h'(\tau)  \geq C_2h(\tau)^p 
 \text{ for } \tau \geq \tau_0,
\\
h(\tau_0)>0,
\\
h'(\tau_0)\geq 0.
\end{array}
\r.
\end{align*}
Since $C_1 -e^{(\beta -1)\tau} >0$ for $\tau \geq \tau_0$, we have $h'(\tau)>0$ by Corollary 3.1 in  \cite{LiZh95}. By $C_1 -e^{(\beta -1)\tau} \leq C_1$ for $\tau \geq \tau_0$ and $h'(\tau)>0$, we have 
\begin{align*}
\l\{
\begin{array}{ll}
e^{ (\beta -1)\tau } C_1^{-1} h''(\tau)  + h'(\tau)  \geq C_2 C_1^{-1} h(\tau)^p 
 \text{ for } \tau \geq \tau_0,
\\
h(\tau_0)>0,
\\
h'(\tau_0)\geq 0.
\end{array}
\r.
\end{align*}
Let $h_1$ satisfy
\begin{align*}
\l\{
\begin{array}{ll}
e^{ (\beta -1)\tau } C_1^{-1} h_1''(\tau)  + h_1'(\tau)  = C_2 C_1^{-1}h_1(\tau)^p 
 \text{ for } \tau \geq \tau_0,
\\
h_1(\tau_0)=\frac{1}{2} h(\tau_0)>0,
\\
h_1'(\tau_0)= 0.
\end{array}
\r.
\end{align*}
Then, Lemma 3.1 in \cite{LiZh95} gives us that $h_1'(\tau)<h'(\tau)$.
Integrating this on $[\tau_0,\tau]$, we get 
\begin{align*}
h_1(\tau) - h_1(\tau_0) < h(\tau) - h(\tau_0),
\end{align*}
which implies $h_1(\tau) < h(\tau)$ for $\tau>\tau_0$ since $h_1(\tau_0)=\frac{1}{2} h(\tau_0)< h(\tau_0)$. Lemma 3.2 in \cite{LiZh95} implies that $h_1''(\tau)>0$ for $\tau\geq \tau_0$. Thus we get
\begin{align*}
C_1^{-1}h_1''(\tau)+h_1'(\tau) \geq C_2 C_1^{-1} h_1(\tau)^p,
\end{align*}
since $e^{ (\beta -1)\tau } \leq 1$. Corollary 3.1' in  \cite{LiZh95} implies $h_1'(\tau)>0$. Therefore, $h_1$ satisfies 
\begin{align*}
\l\{
\begin{array}{ll}
h_1''(\tau)  + h_1'(\tau)  \geq Ch_1(\tau)^p 
 \text{ for } \tau \geq \tau_0,
\\
h_1(\tau_0)>0,
\\
h_1'(\tau_0)= 0.
\end{array}
\r.
\end{align*}
Then, the lifespan of $h_1$ is estimated by $T_{h_1}(\eps) \cleq \eps^{-(p-1)}$ (see Theorem 3.1 in \cite{LiZh95}). Therefore, since $h(\tau)>h_1(\tau)$ for $\tau>\tau_0$, the lifespan of $h$ is also estimated by $T_{h}(\eps) \cleq \eps^{-(p-1)}$. Changing variables, the lifespan of $f$ is estimated by
\begin{align*}
T_{f}(\eps) \leq  \exp \l(C\eps^{-(p-1)}\r).
\end{align*}
\end{proof}

\begin{lemma}
\label{lemA.2}
Let $f$ satisfy the following inequality.
\begin{align}
\l\{
\begin{array}{ll}
(t+1)^{-1} f''(t) + C_1 f'(t) \geq C_2 \frac{f(t)^{p}}{(t+1)\{\log(t+1)+1\}},  & \text{for } t_0\leq  t <T,
\\
f(t_0)>0,
\\
f'(t_0) \geq 0,
\end{array}
\r.
\end{align}
where $C_1 \geq 1$ and $C_2>0$ are constants, $p>1$, and $t_0\geq 0$ satisfies $C_1-e^{-2(e^{\tau_0}-1)}- e^{-2(e^{\tau_0}-1)} e^{-\tau_0} >0$, where $\tau_0=\log\{ \log (t_0+1)+1\}$. Then, we have
\begin{align*}
T \leq \exp( \exp (C \eps^{-(p-1)})).
\end{align*}
\end{lemma}

\begin{proof}
Substituting $e^{e^{t}-1}-1$ into $t$ in the ordinary differential inequality, we obtain
\begin{align}
\label{eqA.4}
e^{-e^{t}+1} f''(e^{e^{t}-1}-1) +C_1 f'(e^{e^{t}-1}-1) \geq C_2 \frac{f(e^{e^{t}-1}-1)^p}{e^{e^{t}-1}e^{t}}
\end{align}
for any $t \geq \log \{ \log(t_0+1)+1 \}$. 
We set $h(t):=f(e^{e^{t}-1}-1)$. Then, we have
\begin{align*}
h'(t)&=f'(e^{e^{t}-1}-1)e^{e^{t}-1}e^{t},
\\
h''(t)&=f''(e^{e^{t}-1}-1)e^{2(e^{t}-1)} e^{2t}+(e^{t}+1) h'(t).
\end{align*}
Therefore, by \eqref{eqA.4}, we have
\begin{align*}
e^{-2(e^{t}-1)}e^{-t} h''(t)  +\{C_1-e^{-2(e^{t}-1)}- e^{-2(e^{t}-1)} e^{-t}\} h'(t)  \geq C_2 h(t)^p
\end{align*}
for any $t \geq \log \{ \log(t_0+1)+1 \}$. Since 
\begin{align*}
0 &< f(t_0) = h( \log \{ \log(t_0+1)+1 \}),
\\
0 & \leq f'(t_0) = h'( \log \{ \log(t_0+1)+1 \})(t_0+1) \{ \log(t_0+1)+1 \} 
\end{align*}
mean that $h(\tau_0)>0$ and $h'(t_0) \geq0$, where $\tau_0:= \log \{ \log(t_0+1)+1 \}$, we obtain
\begin{align*}
\l\{
\begin{array}{ll}
e^{-2(e^{\tau}-1)}e^{-\tau} h''(\tau)  +\{C_1-e^{-2(e^{\tau}-1)}- e^{-2(e^{\tau}-1)} e^{-\tau}\} h'(\tau)  \geq C_2 h(\tau)^p 
 \text{ for } \tau \geq \tau_0,
\\
h(\tau_0)>0,
\\
h'(\tau_0)\geq 0.
\end{array}
\r.
\end{align*}
Since $C_1-e^{-2(e^{\tau}-1)}- e^{-2(e^{\tau}-1)} e^{-\tau} >0$ for $\tau \geq \tau_0$, we have $h'(\tau)>0$ by Corollary 3.1 in  \cite{LiZh95}. By $C_1-e^{-2(e^{\tau}-1)}- e^{-2(e^{\tau}-1)} e^{-\tau} \leq C_1$ for $\tau \geq \tau_0$ and $h'(\tau)>0$, we have 
\begin{align*}
\l\{
\begin{array}{ll}
e^{-2(e^{\tau}-1)}e^{-\tau} C_1^{-1} h''(\tau)  + h'(\tau)  \geq C_2 C_1^{-1}h(\tau)^p 
 \text{ for } \tau \geq \tau_0,
\\
h(\tau_0)>0,
\\
h'(\tau_0)\geq 0.
\end{array}
\r.
\end{align*}
Let $h_1$ satisfy
\begin{align*}
\l\{
\begin{array}{ll}
e^{-2(e^{\tau}-1)}e^{-\tau}C_1^{-1} h_1''(\tau)  + h_1 '(\tau)  = C_2 C_1^{-1}h_1(\tau)^p 
 \text{ for } \tau \geq \tau_0,
\\
h_1(\tau_0)=\frac{1}{2} h(\tau_0)>0,
\\
h_1'(\tau_0)= 0.
\end{array}
\r.
\end{align*}
Then, Lemma 3.1 in \cite{LiZh95} gives us that $h_1'(\tau)<h'(\tau)$.
Integrating this on $[\tau_0,\tau]$, we get 
\begin{align*}
h_1(\tau) - h_1(\tau_0) < h(\tau) - h(\tau_0),
\end{align*}
which implies $h_1(\tau) < h(\tau)$ for $\tau>\tau_0$ since $h_1(\tau_0)=\frac{1}{2} h(\tau_0)< h(\tau_0)$. Lemma 3.2 in \cite{LiZh95} implies that $h_1''(\tau)>0$ for $\tau\geq \tau_0$. Thus we get
\begin{align*}
C_1^{-1} h_1''(\tau)+h_1'(\tau) \geq C_2 C_1^{-1} h_1(\tau)^p,
\end{align*}
since $e^{-2(e^{\tau}-1)}e^{-\tau}\leq 1$. Corollary 3.1' in \cite{LiZh95} implies $h_1'(\tau)>0$. Therefore, $h_1$ satisfies 
\begin{align*}
\l\{
\begin{array}{ll}
h_1''(\tau)  + h_1'(\tau)  \geq Ch_1(\tau)^p 
 \text{ for } \tau \geq \tau_0,
\\
h_1(\tau_0)>0,
\\
h_1'(\tau_0)= 0.
\end{array}
\r.
\end{align*}
Then, the lifespan of $h_1$ is estimated by $T_{h_1}(\eps) \cleq \eps^{-(p-1)}$ (see Theorem 3.1 in \cite{LiZh95}). Therefore, since $h(\tau)>h_1(\tau)$ for $\tau>\tau_0$, the lifespan of $h$ is also estimated by $T_{h}(\eps) \cleq \eps^{-(p-1)}$. Changing variables, the lifespan of $f$ is estimated by
\begin{align*}
T_{f}(\eps) \leq \exp \l(\exp \l(C\eps^{-(p-1)}\r) \r).
\end{align*}
\end{proof}


%



We give the proof of Lemma \ref{lem2.2}. 

\begin{proof}[Proof of Lemma \ref{lem2.2}]
The proof of (i) and (ii) can be seen in Lin--Nishihara--Zhai. However, we give the proofs for the reader's convenience. 

(i). By the explicit formula of $g$ and the l'H\^{o}pital theorem, we obtain
\begin{align*}
\lim_{t \to \infty} b(t)g(t) 
&=\lim_{t \to \infty} \frac{\int_{0}^{\infty} e^{-\int_{0}^{\tau} b(s)ds}d\tau- \int_{0}^{t} e^{-\int_{0}^{\tau} b(s)ds}d\tau}{b(t)^{-1}e^{-\int_{0}^{t} b(s) ds}}
\\
&=\lim_{t \to \infty} \frac{-e^{-\int_{0}^{t} b(s)ds}d\tau}{-b(t)^{-2}b'(t)e^{-\int_{0}^{t} b(s) ds}-e^{-\int_{0}^{t} b(s) ds}}
\\
&=\lim_{t \to \infty} \frac{1}{b(t)^{-2}b'(t)+1}
\\
&=1,
\end{align*}
where we have used $|b(t)^{-2}b'(t)| \cleq (t+1)^{\beta-1} \to 0$ as $t \to \infty$ when $\beta<1$.

(ii). Since $b(t)g(t)>0$ for any $t \in[0,\infty)$ and (i) hold, $bg$ has the positive minimum $m$ and the maximum $M$. Therefore, 
we obtain $m \leq b(t)g(t) \leq M$ for any $t \in[0,\infty)$. This implies the statement. 

(iii). Since $(b(t)^{-1})'=-b'(t)b(t)^{-2}$ and $g'(t)=b(t)g(t)-1$, we have
\begin{align*}
\frac{g'(t)}{(b(t)^{-1})'} 
= \frac{b(t)g(t)-1}{-b'(t)b(t)^{-2}}.
\end{align*}
By the explicit formula of $g$ and the l'H\^{o}pital lemma, we have
\begin{align*}
&\lim_{t \to \infty} \frac{b(t)g(t)-1}{-b'(t)b(t)^{-2}}
\\
&=\lim_{t \to \infty}  \frac{b(t) e^{\int_{0}^{t} b(s) ds} \l( \int_{0}^{\infty} e^{-\int_{0}^{\tau} b(s)ds}d\tau- \int_{0}^{t} e^{-\int_{0}^{\tau} b(s)ds}d\tau \r) -1 }{-b'(t)b(t)^{-2}} 
\\
&=\lim_{t \to \infty}  \frac{ \int_{0}^{\infty} e^{-\int_{0}^{\tau} b(s)ds}d\tau- \int_{0}^{t} e^{-\int_{0}^{\tau} b(s)ds}d\tau  -b(t)^{-1} e^{-\int_{0}^{t} b(s) ds}  }{-b'(t)b(t)^{-3} e^{-\int_{0}^{t} b(s) ds}}
\\
&=\lim_{t \to \infty}  \frac{- e^{-\int_{0}^{t} b(s)ds}  + b'(t)b(t)^{-2} e^{-\int_{0}^{t} b(s) ds} + e^{-\int_{0}^{t} b(s) ds}  }{-b''(t)b(t)^{-3} e^{-\int_{0}^{t} b(s) ds} +3b'(t)^2b(t)^{-4} e^{-\int_{0}^{t} b(s) ds} +b'(t)b(t)^{-2} e^{-\int_{0}^{t} b(s) ds}}
\\
&=\lim_{t \to \infty}  \frac{b'(t)b(t)^{-2}  }{-b''(t)b(t)^{-3}  +3b'(t)^2b(t)^{-4}  +b'(t)b(t)^{-2} }
\\
&=\lim_{t \to \infty}  \frac{1 }{\frac{-b''(t)}{b'(t)b(t) }  +3\frac{b'(t)}{b(t)^{2}} +1 }.
\end{align*}
By the assumption (B2)--(B4) and $\beta<1$, we get
\begin{align*}
\l| \frac{-b''(t)}{b'(t)b(t) } \r|
& \cleq \frac{(t+1)^{-\beta-2}}{(t+1)^{-\beta-1}  (t+1)^{-\beta}}=(t+1)^{\beta-1} \to 0,
\\
\l| \frac{b'(t)}{b(t)^{2}}  \r|
& \cleq \frac{ (t+1)^{-\beta-1}}{(t+1)^{-2\beta}}=(t+1)^{\beta-1} \to 0,
\end{align*}
as $t \to \infty$. Thus, we obtain
\begin{align*}
\lim_{t \to \infty} \frac{b(t)g(t)-1}{-b'(t)b(t)^{-2}}
=\lim_{t \to \infty}  \frac{1 }{\frac{-b''(t)}{b'(t)b(t) }  +3\frac{b'(t)}{b(t)^{2}} +1 }=1.
\end{align*}

By this and $(b(t)^{-1})'=-b'(t)b(t)^{-2}\ceq (t+1)^{\beta-1}$, we obtain $|g'(t)| \ceq (t+1)^{\beta-1}$ for large $t>0$.

(iv). $G$ is increasing since $g(t)>0$ for any $t\in[0,\infty)$. Moreover, since $g\thickapprox b^{-1}$, we have
\begin{align*}
G(t)+1
& \thickapprox \int_{0}^{t} b(s)^{-1} ds +1
\\
& =  \int_{0}^{t} (s+1)^{\beta} ds +1
\\
& \thickapprox 
\begin{cases}
\log(t+1)+1, & \text{ if } \beta=-1,
\\
(t+1)^{\beta+1}, & \text{ if } \beta \in (-1,1). 
\end{cases} 
\end{align*}

(v). $\Gamma$ is increasing since $g(t)>0$ for any $t\in[0,\infty)$. Moreover, since $g\thickapprox b^{-1}$, we have
\begin{align*}
\Gamma(t)+1
 \thickapprox \int_{0}^{t} b(s) ds +1
 =  \int_{0}^{t} (s+1)^{-\beta} ds +1
 \thickapprox  (t+1)^{-\beta+1}.
\end{align*}
\end{proof}


\end{document}